\documentclass[12pt]{amsart}

\usepackage{amssymb}
\usepackage{amsthm}

\usepackage[usenames]{color}

\usepackage[margin=1in]{geometry}
\usepackage[foot]{amsaddr}

\usepackage{hyperref}

\theoremstyle{plain}
\newtheorem{proposition}{Proposition}[section]
\newtheorem{theorem}[proposition]{Theorem}
\newtheorem{lemma}[proposition]{Lemma}
\newtheorem{corollary}[proposition]{Corollary}
\theoremstyle{definition}

\newtheorem{definition}[proposition]{Definition}
\newtheorem{observation}[proposition]{Observation}
\theoremstyle{remark}
\newtheorem{remark}[proposition]{Remark}

\DeclareMathOperator{\Aut}{Aut}
\DeclareMathOperator{\dimension}{dim}

\DeclareMathOperator{\SL}{SL}
\DeclareMathOperator{\GL}{GL}
\DeclareMathOperator{\SO}{SO}
\DeclareMathOperator{\SU}{SU}
\DeclareMathOperator{\PSL}{PSL}
\DeclareMathOperator{\Hom}{Hom}
\DeclareMathOperator{\End}{End}

\DeclareMathOperator{\Cc}{\mathcal{C}}
\DeclareMathOperator{\Fc}{\mathcal{F}}

\DeclareMathOperator{\Lc}{\mathcal{L}}

\DeclareMathOperator{\Oc}{\mathcal{O}}
\DeclareMathOperator{\Qc}{\mathcal{Q}}

\DeclareMathOperator{\Cb}{\mathbb{C}}
\DeclareMathOperator{\Db}{\mathbb{D}}

\DeclareMathOperator{\Hb}{\mathbb{H}}

\DeclareMathOperator{\Kb}{\mathbb{K}}
\DeclareMathOperator{\Nb}{\mathbb{N}}
\DeclareMathOperator{\Pb}{\mathbb{P}}
\DeclareMathOperator{\Rb}{\mathbb{R}}
\DeclareMathOperator{\Zb}{\mathbb{Z}}

\newcommand{\abs}[1]{\left|#1\right|}

\newcommand{\norm}[1]{\left\|#1\right\|}

\begin{document}

\title{Rigidity of complex convex divisible sets}
\author{Andrew M. Zimmer}\address{Department of Mathematics, University of Michigan, Ann Arbor, MI 48109.}
\email{aazimmer@umich.edu}
\date{\today}

\begin{abstract} An open convex set in real projective space is called divisible if there exists a discrete group of projective automorphisms which acts co-compactly. There are many examples of such sets and a theorem of Benoist implies that many of these examples are strictly convex, have $C^1$ boundary, and have word hyperbolic dividing group. In this paper we study a notion of convexity in complex projective space and show that the only divisible complex convex sets with $C^1$ boundary are the projective balls. 
\end{abstract}

\maketitle

\section{Introduction}

A set $\Omega \subset \Pb(\Rb^{d+1})$ is called \emph{convex} if whenever $L$ is a real projective line then  $\Omega \cap L$ is connected. A convex set $\Omega \subset \Pb(\Rb^{d+1})$ called \emph{proper} if $\overline{\Omega}$ does not contain any projective lines. A proper convex open set $\Omega$ is called \emph{divisible} if there is a discrete group $\Gamma \leq \PSL(\Rb^{d+1})$ preserving $\Omega$ such that $\Gamma \backslash \Omega$ is compact. 

The most basic examples of divisible convex open sets come from hyperbolic geometry. Let $\Omega \subset \Pb(\Rb^{d+1})$ be the projective ball
\begin{align*}
\Omega = \left\{ [1 : x_1 : x_2 : \dots : x_d ] : \sum x_i^2  < 1\right\}.
\end{align*}
Then $\SO^+(1,d) \leq \PSL(\Rb^{d+1})$ acts properly and transitively on $\Omega$ and in particular if $\Gamma \leq \SO^+(1,d)$ is a cocompact lattice then $\Gamma$ divides $\Omega$. There is also a natural metric, the \emph{Hilbert metric} $d_{\Omega}$, on a proper open convex set $\Omega \subset \Pb(\Rb^{d+1})$ and when $\Omega$ is a projective ball $(\Omega, d_{\Omega})$ is isometric to real hyperbolic $d$-space. 

Additional examples of divisible convex sets can be constructed by deformations. To be more precise: given a finitely generated group $\Gamma_0$ one can consider the space
\begin{align*}
\Fc_{\Gamma_0} = \{  \rho \in & \Hom(\Gamma_0, \PSL(\Rb^{d+1}))  \text{ is faithful and the image divides} \\
& \text{an open properly convex set in $\Pb(\Rb^{d+1})$} \}.
\end{align*}
By theorems of Kozul~\cite{K1968} and Thurston (see~\cite[Theorem 3.1]{G1988}) this set is open in $\Hom(\Gamma_0, \PSL(\Rb^{d+1}))$ and assuming that $\Fc_{\Gamma_0}$ contains a strongly irreducible representation Benoist proved that it is closed~\cite{B2005}. Johnson and Millson~\cite{JM1987} proved that for certain cocompact lattices $\Gamma_0 \leq \SO^+(1,d)$ it is possible to deform the inclusion map $\rho_0:\Gamma_0 \rightarrow \SO^+(1,d) \rightarrow \PSL(\Rb^{d+1})$ to obtain a family of representations $\rho_s : \Gamma_0 \rightarrow \PSL(\Rb^{d+1})$ such that $\rho_s(\Gamma_0)$ is Zariski dense for all $s \neq 0$. This implies that $\Fc_{\Gamma_0}$ will be a non-trivial deformation space of $\Gamma_0$. For surface groups Goldman has given an explicit parametrization of this deformation space~\cite{G1990}. Thus the theory of divisible convex sets in real projective space leads very naturally to a type of ``Higher Teichm\"{u}ller theory.'' 

Using a bending construction Kapovich~\cite{K2007} has produced divisible convex sets such that the dividing groups are not quasi-isometric to any real hyperbolic lattice. Moreover, the dividing groups that Kapovich produces are word hyperbolic and the resulting compact manifolds $\Gamma \backslash \Omega$ are the Gromov-Thurston examples~\cite{GT1987} of compact manifolds admitting a Riemannian metric of negative curvature but not one of constant negative curvature. 

Divisible real convex sets also have enough structure so that it is possible to prove interesting and non-trivial theorems. For instance, Benoist has proven the following:

\begin{theorem}\cite[Theorem 1.1]{B2004}
Suppose $\Omega \subset \Pb(\Rb^{d+1})$ is a proper convex open set. If $\Omega$ is divisible by $\Gamma \leq \PSL(\Rb^{d+1})$ then 
\begin{center}
$\Omega$ is strictly convex $\Leftrightarrow$ $\partial \Omega$ is $C^1$ $\Leftrightarrow$ $\Gamma$ is word hyperbolic.
\end{center}
\end{theorem}

Now it is well known that the rank one symmetric space $\Kb \Hb^d$ can be realized as a ball in $\Pb(\Kb^{d+1})$ with a natural metric. Thus it is reasonable to ask if any of the above theory can be generalized to other projective spaces.  

In this paper, we consider versions of the definitions above in complex projective space and prove a result showing that the theory in the complex setting is much more rigid. Already, Goldman observed that by Weil rigidity the only deformations of a complex hyperbolic lattice $\Gamma \leq \SU(1,d) \leq \SL(\Cb^{d+1})$ are by conjugation (this argument can be found in Goldman's MathSciNet review of~\cite{K2011}). 

There are many different types of convexity for sets in $\Pb(\Cb^{d+1})$ (see~\cite{APS2004} or~\cite{H2007}), but perhaps the most natural is $\Cb$-convexity. A set $\Omega \subset \Pb(\Cb^{d+1})$ is called \emph{$\Cb$-convex} if whenever $L$ is a complex projective line then $L \cap \Omega$ and $L \setminus \Omega \cap L$ are connected. This definition has many of the properties one would expect: the projective image of any $\Cb$-convex set is $\Cb$-convex and the complex dual $\Omega^*$ of $\Omega$ will be $\Cb$-convex (see for instance~\cite{APS2004}).

A $\Cb$-convex set is called \emph{proper} if $\overline{\Omega}$ does not contain a complex projective line. A proper $\Cb$-convex open set $\Omega$ is called \emph{divisible} if there exists a discrete group $\Gamma \leq \PSL(\Cb^{d+1})$ preserving $\Omega$ such that $\Gamma \backslash \Omega$ is compact. With these definitions we are ready to state our main result.

\begin{theorem} 
\label{thm:main}
Suppose $\Omega$ is a divisible  proper $\Cb$-convex open set with $C^1$ boundary. Then $\Omega$ is a projective ball.
\end{theorem}

\begin{remark} \
\begin{enumerate}
\item We say a set $\Omega \subset \Pb(\Cb^{d+1})$ is a \emph{projective ball} when there exists a basis of $\Cb^{d+1}$ such that with respect to this basis $\Omega = \{ [1:z_1:\dots : z_d] : \sum \abs{z_i}^2 <1 \}$. A \emph{projective disk} is a projective ball in $\Pb(\Cb^{2})$.
\item By $C^1$ boundary we mean that $\partial \Omega$ is a $C^1$ embedded submanifold of $\Pb(\Cb^{d+1})$.
\item The examples of real divisible convex sets constructed by deforming a real hyperbolic lattice and the examples constructed by Kapovich have word hyperbolic dividing groups. Then Benoist's theorem  implies that all these examples have $C^1$ boundary and so in real projective geometry there are many examples of divisible proper convex open sets with $C^1$ boundary.
\end{enumerate}
 \end{remark}

There are at least two other well known notions of convexity for sets in $\Pb(\Cb^{d+1})$. A set $\Omega$ is called \emph{linearly convex} if for all $p \in \Pb(\Cb^{d+1})\setminus \Omega$ there exists a complex hyperplane $H$ containing $p$ which does not intersect $\Omega$. An open set $\Omega$ is \emph{weakly linearly convex} if for all $p \in \partial \Omega$ there exists a complex hyperplane $H$ containing $p$ which does not intersect $\Omega$.  For an open set we have the following implications:
\begin{align*}
\Cb\text{-convex} \Rightarrow \text{linearly convex} \Rightarrow \text{weakly linear convex.}
\end{align*}
For a proof see~\cite[Theorem 3.29]{APS2004}. When $\Omega$ is an open set and $\partial \Omega$ is $C^1$, then all three definitions coincide (see for instance~\cite[Corollary 2.5.6]{APS2004}). Thus we have the following corollary of Theorem~\ref{thm:main}:

\begin{corollary}
Suppose $\Omega$ is a divisible proper weakly linearly convex open set with $C^1$ boundary. Then $\Omega$ is a projective ball.
\end{corollary}

\subsection{Outline of proof} A powerful tool for studying divisible convex sets in real projective geometry is the Hilbert metric (see for instance~\cite{B2004} or~\cite{V1970}). Following this trend, one of the main tools in our proof is a Hilbert metric $d_\Omega$ on a linearly convex set $\Omega \subset \Pb(\Cb^{d+1})$. This metric was discovered by Dubois~\cite{D2009}. When $\Omega$ is proper this Hilbert metric will be complete and invariant under projective transformations. There are a number of other well known metrics (for example the Kobayashi metric or the Bergman metric) but the Hilbert metric has several useful properties:
\begin{enumerate}
\item if $W$ is a projective subspace then the inclusion map $W \cap \Omega \hookrightarrow \Omega$ induces an isometric embedding $(W \cap \Omega, d_{W \cap \Omega}) \hookrightarrow (\Omega, d_{\Omega})$ (see Theorem~\ref{thm:dubois}),
\item the behavior of the metric near the boundary is closely related to the geometry of the boundary (see Proposition~\ref{prop:bd_behav}), 
\item for $\varphi \in \PSL(\Cb^{d+1})$ preserving $\Omega$ it is possible to estimate the translation distance $d_{\Omega}(\varphi y, y)$ in terms of $\norm{\varphi}$.
(see Proposition~\ref{prop:trans_dist}). 
\end{enumerate}
Despite these properties the Hilbert metric does have one serious downfall: it is rarely geodesic (see Theorem~\ref{thm:not_geod}). But when the boundary is $C^1$ the Hilbert metric will be a quasi-geodesic metric (see Theorem~\ref{thm:quasi_geod}) and this is good enough for the  \v{S}varc-Milnor Lemma (see Theorem~\ref{thm:fundgeomgp}).

Section~\ref{sec:prelim} is devoted to preliminaries concerning $\Cb$-convex sets. We also briefly discuss quasi-geodesic metric spaces. In Section~\ref{sec:Hilbert}, we will re-introduce the Hilbert metric and establish the properties mentioned above. 

The proof of Theorem~\ref{thm:main} occupies the rest of the paper. In Section~\ref{sec:bi_prox_or_unip}, we will  prove that each $\gamma \in \Gamma$ is either bi-proximal or ``almost unipotent.'' In Section~\ref{sec:exist_bi}, we will prove that $\Gamma$ contains a bi-proximal element. In Section~\ref{sec:constr_auto}, we will use the existence of bi-proximal elements preserving $\Omega$  to construct many additional automorphisms of $\Omega$. In Sections~\ref{sec:str} and~\ref{sec:smth}, we will use these additional projective automorphisms to show that $\Omega$ is strictly $\Cb$-convex and that the boundary is $C^\infty$. Finally in Section~\ref{sec:complete}, these latter two conditions are exploited to prove that $\Omega$ is a projective ball using a scaling argument. 

\subsection{Prior Results} As mentioned above there is a rich theory of convex divisible sets in real projective space. For additional details about this theory see the survey articles by Benoist~\cite{B2008}, Guo~\cite{G2013}, Quint~\cite{Q2010}, and Marquis~\cite{M2013}.

There are also several related rigidity results coming from the complex analysis community. One remarkable theorem is the ball theorem of Rosay~\cite{R1979} and Wong~\cite{W1977}:

\begin{theorem}
Suppose $\Omega \subset \Cb^d$ is a bounded strongly pseudo-convex domain. If the space of holomorphic automorphisms of $\Omega$ is non-compact then $\Omega$ is bi-holomorphic to a ball.
\end{theorem}

By \emph{strongly psuedo-convex} we mean that $\Omega$ has $C^2$ boundary and the Levi-form at each point in the boundary is positive definite. There are more general versions of the ball theorem requiring only that the boundary is strongly psuedo-convex at an orbit accumulation point. We refer the reader to the survey articles~\cite{IK1999} and~\cite{K2013} for more details. 

We should also note that Soci{\'e}-M{\'e}thou has proven a version of the ball theorem for convex sets in real projective space.

\begin{theorem}\cite{SM2002}
Suppose $\Omega \subset \Pb(\Rb^{d+1})$ is an open proper strongly convex set. If the space of projective automorphisms of $\Omega$ is non-compact then $\Omega$ is a projective ball. 
\end{theorem}

By \emph{strongly convex set} we mean that $\Omega$ has $C^2$ boundary and the Hessian at each point in the boundary is positive definite. In the real projective world, Benoist showed that rigidity still holds if the boundary regularity is slightly relaxed but at the cost of assuming the existence of a dividing group.

\begin{theorem}\cite[Theorem 1.3]{B2004}
Suppose $\Omega \subset \Pb(\Rb^{d+1})$ is an open, proper, divisible convex set. If $\partial \Omega$ is $C^{1+\alpha}$ for all $\alpha \in [0,1)$ then $\Omega$ is a projective ball.
\end{theorem}

Thus for convex sets in real projective space we see that the theory is rigid for sets with $C^2$ boundary but rich in examples for sets with $C^1$ boundary. 

Another remarkable theorem is due to Frankel.

\begin{theorem}\cite{F1989}
Suppose $\Omega \subset \Cb^d$ is a bounded convex (in the usual sense) domain and there exists a discrete group $\Gamma$ of holomorphic automorphisms of $\Omega$ such that $\Gamma \backslash \Omega$ is compact. Then $\Omega$ is a symmetric domain.
\end{theorem}

It is clear that if $\Omega \subset \Pb(\Cb^{d+1})$ is convex (in the usual sense) in some affine chart then $\Omega$ will be $\Cb$-convex. Moreover there exists bounded $\Cb$-convex sets in $\Cb^d$ which are not bi-holomorphic to a convex set~\cite{NPZ2008}.  In particular the main theorem of this paper weakens the hypothesis of Frankel's result in one direction while strengthening the hypothesis in two directions: assuming additional boundary regularity and assuming the dividing group acts projectively instead of holomorphically. We should also mention that $\Cb$-convexity is invariant under projective transformations while ordinary convexity is not. Thus for proving rigidity results about the group of projective automorphisms of a domain it seems more natural to look at $\Cb$-convex domains.

\subsection*{Acknowledgements}

I would like to thank David Barrett, Daniel Burns, Gopal Prasad, and my advisor Ralf Spatzier for many helpful conversations. This material is based upon work supported by the National Science Foundation under Grant Number NSF 1045119.

\section{Preliminaries}\label{sec:prelim}

\subsection{Notations} Suppose $V$ is a complex vector space. If $v \in V$ we will let $[v] \in \Pb(V)$ denote the equivalence class of $v$. Similarly, given a subspace $W \leq V$ we will let $[W]$ be the image of $W$ in $\Pb(V)$. Finally, given $\psi \in \SL(V)$ we will let $[\psi] \in \PSL(V)$ denote the equivalence class of $\psi$. 

If $\norm{ \cdot }$ is a norm on $V$ then we will also let  $\norm{ \cdot}$ denote the induced norm on $V^*$ and the operator norm on $\End(V)$.

Given an open set $\Omega \subset \Pb(V)$ we will let $\Aut(\Omega) \leq \PSL(V)$ denote the group of projective transformations $\varphi \in \PSL(V)$ such that $\varphi(\Omega)=\Omega$. We will also let $\Aut_0(\Omega)$ denote the connected component of $\Aut(\Omega)$. Notice that $\Aut_0(\Omega)$ is a closed connected subgroup of $\PSL(V)$ and hence $\Aut_0(\Omega)$  is a Lie subgroup of $\PSL(V)$.

\subsection{Complex tangent hyperplanes} Suppose $\Omega \subset \Pb(\Cb^{d+1})$ is an open set, then we say a complex hyperplane $H$ is tangent to $\Omega$ at $p \in \partial \Omega$ if $p \in H$ and $\Omega \cap H = \emptyset$. If $H$ is tangent to some point $p \in \partial \Omega$, then we say $H$ is a \emph{complex tangent hyperplane of $\Omega$}. 

If $\Omega$ is linearly convex open set, then every point $p \in \partial \Omega$ is contained in at least one complex tangent hyperplane. If, in addition, $\partial \Omega$ is $C^1$ there is an unique complex tangent hyperplane through each boundary point. In this case, if $x \in \partial \Omega$ we will denote the unique complex tangent hyperplane at $x$ by $T_x^{\Cb} \partial \Omega$. Finally we have the following observation 

\begin{observation}
Suppose $\Omega$ is a linearly convex open set with $C^1$ boundary. If $P$ is a projective linear subspace such that $x \in P \cap \partial \Omega$ and $P \cap \Omega = \emptyset$ then $P \subset T_x^{\Cb} \partial \Omega$.
\end{observation}

\subsection{The complex dual:} Given a subset $\Omega \subset \Pb(V)$, the \emph{complex dual $\Omega^*$ of $\Omega$} is the set
\begin{align*}
\Omega^* = \{ f \in \Pb(V^*) : f(x) \neq 0 \text{ for all } x \in \Omega\}.
\end{align*}
When $\Omega$ is an open $\Cb$-convex set then $\Omega^*$ will be $\Cb$-convex as well (see for instance~\cite[Theorem 2.3.9]{APS2004}). So if $\Omega$ is an open $\Cb$-convex set then $\Omega^*$ will be a compact $\Cb$-convex set. 

In the real projective case, if $\Cc \subset \Pb(\Rb^{d+1})$ is an open proper convex set then $\Cc^*$ will have non-empty interior. In the complex case, the story is more complicated but we do have the following.

\begin{proposition}
\label{prop:zar_den}\cite[Proposition 2.3.10]{APS2004}
If $\Omega \subset \Pb(V)$ is a proper $\Cb$-convex open set, then $\Omega^*$ is not contained in a hyperplane.
\end{proposition}

\subsection{Quasi-geodesic metric spaces} As mentioned in the introduction, the proof of Theorem~\ref{thm:main} makes use of a Hilbert metric defined for linearly convex sets. In Section~\ref{sec:Hilbert} we will show that $(\Omega,d_{\Omega})$ is a quasi-geodesic metric space when $\Omega$ satisfies the hypothesis of Theorem~\ref{thm:main}. In this section we first recall the definitions of quasi-geodesics and quasi-geodesic metric spaces. Then we will state an important property of such spaces.

If $(X,d_X)$ and $(Y,d_Y)$ are metric spaces, a map $f: X \rightarrow Y$ is called a \emph{$(A,B)$-quasi isometric embedding} if 
\begin{align*}
 \frac{1}{A}d_X(x_1,x_2) -B \leq d_Y(f(x_1),f(x_2)) \leq A d_X(x_1,x_2) + B
\end{align*}
for all $x_1,x_2 \in X$. An embedding becomes an \emph{isomorphism} when there exists some $R>0$ such that $Y$ is contained in the $R$-neighborhood of $f(X)$.

The set of integers $\Zb$ has a natural metric: $d_{\Zb}(i,j) = \abs{i-j}$ and a $(A,B)$-quasi isometric embedding of $\{1,\dots,N\}$ is called a \emph{$(A,B)$-quasi geodesic segment}. A metric space $(X,d)$ is called \emph{$(A,B)$-quasi-geodesic} if for all $x,y \in X$ there exists a $(A,B)$-quasi-geodesic segment whose image contains $x$ and $y$. If $(X,d)$ is $(A,B)$-quasi geodesic for some $A$ and $B$, then $(X,d)$ is called a \emph{quasi-geodesic metric space}.

We now observe that the \v{S}varc-Milnor Lemma is true for quasi-geodesic metric spaces. More precisely: given a finitely generated group $\Gamma$ and a set of generators $S=\{s_1, \dots, s_k\}$ define the word metric $d_S$ on $\Gamma$ by 
\begin{align*}
d_S(\gamma_1,\gamma_2) = \inf \{ N : \gamma_2^{-1}\gamma_1 = s_{i_1} \dots s_{i_N} \},
\end{align*}
we then have the following:

\begin{theorem}
\label{thm:fundgeomgp}
Suppose $(X,d)$ is a proper quasi-geodesic metric space and $\Gamma$ is a group acting on $(X,d)$ by isometries. If the action is properly discontinuous and cocompact then $\Gamma$ is finitely generated. Moreover, if $S \subset \Gamma$ is a finite generating set and $x_0 \in X$ then the map $\gamma \in \Gamma \rightarrow \gamma \cdot x_0 \in X$ is a quasi-isometry of $(\Gamma,d_S)$ and $(X,d)$.
\end{theorem}

\begin{proof}
The proof of the theorem for geodesic metric spaces given in~\cite[Chapter IV, Theorem 23]{dlP2000}) can be extended to quasi-geodesic spaces essentially verbatim.
\end{proof}

\section{The Hilbert Metric}\label{sec:Hilbert}

In this section we recall a metric introduced by Dubois~\cite{D2009} and prove some basic results. For a proper open linearly convex set $\Omega \subset \Pb(V)$ define the \emph{Hilbert metric} $d_\Omega$ on $\Omega$ by 
\begin{align}
\label{eq:hil}
d_{\Omega}(v,w) = \max_{f,g \in \Omega^*} \log \left( \frac{\abs{f(v)g(w)}}{\abs{f(w)g(v)}} \right).
\end{align}
This is clearly well defined and satisfies the triangle inequality. Moreover, when $\Omega$ is proper $d_{\Omega}(v,w) = 0$ if and only if $v=w$ \cite[Lemma 2.1]{D2009}. We note that the Hilbert metric for a convex set in real projective space is usually defined in terms of the cross ratio, but it is straightforward to check that it can also be defined as in Equation~(\ref{eq:hil}).

With $v, w \in \Pb(V)$ fixed, the map $(f,g) \in \Pb(V^*) \times \Pb(V^*) \rightarrow \frac{\abs{f(v)g(w)}}{\abs{f(w)g(v)}} \in \Rb$ is an open map and so it is enough to consider $f,g \in \partial \Omega^*$:
\begin{align*}
d_{\Omega}(v,w) = \max_{f,g \in \partial\Omega^*} \log \left( \frac{\abs{f(v)g(w)}}{\abs{f(w)g(v)}} \right).
\end{align*}
Moreover Dubois proved the following

\begin{theorem}\cite[Lemma 2.1, Lemma 2.2]{D2009}\label{thm:dubois}
If $\Omega$ is a proper linearly convex open set  then $d_{\Omega}$ is a complete metric on $\Omega$. Moreover if $L \subset \Pb(V)$ is a complex projective line and $x,y \in L \cap \Omega$ then $d_{\Omega}(x,y) = d_{\Omega \cap L}(x,y)$.
\end{theorem}

Notice that if $\gamma \in \Aut(\Omega)$ then $^t\gamma$ will preserve $\Omega^*$ and hence we immediately see the following:

\begin{lemma}
Suppose $\Omega$ is a proper linearly convex open set . If $\gamma \in \Aut(\Omega)$ then the action of $\gamma$ on $\Omega$ is an isometry with respect to the Hilbert metric.
\end{lemma}

We also have the following lemma showing that the Hilbert metric generates the standard topology on $\Omega$.

\begin{lemma}
\label{lem:same_topo}
Suppose $\Omega \subset \Pb(V)$ is a proper linearly convex open set . Then 
\begin{enumerate}
\item for any $p \in \Omega$ and $R >0$ the closed ball 
\begin{align*}
B_R(p) = \{ q \in \Omega : d_{\Omega}(p,q) \leq R\}
\end{align*}
is a compact set in $\Omega$ (with respect to the subspace topology), and
\item the subspace topology on $\Omega \subset \Pb(V)$ and the topology on $\Omega$ induced by $d_{\Omega}$ coincide.
\end{enumerate}
\end{lemma}

\begin{proof}
The map $F : \Omega \times \Omega \times \Omega^* \times \Omega^* \rightarrow \Rb$ given by 
\begin{align*}
F(v,w,f,g) = \log \left( \frac{\abs{f(v)g(w)}}{\abs{f(w)g(v)}} \right)
\end{align*}
is clearly continuous in the subspace topology. Then since $\Omega^*$ is compact $d_{\Omega} : \Omega \times \Omega \rightarrow \Rb$ is continuous in the subspace topology. Thus if $x_n$ converges to $x$ in the subspace topology then $d_{\Omega}(x_n,x) \rightarrow 0$ and hence $x_n$ converges to $x$ in the topology induced by $d_{\Omega}$.

We now show $B_R(p)$ is compact. Since $d_{\Omega}$ is continuous with respect to the subspace topology, $B_R(p)$ is closed in $\Omega$ with respect to the subspace topology. To see that $B_R(p)$ is compact it is enough to establish the following: if $\{q_n\}_{n \in \Nb}  \subset \Omega$ is a sequence such that $q_n \rightarrow y \in \partial \Omega$  then $d_{\Omega}(p,q_n) \rightarrow \infty$. Let $f \in \Omega^*$ be such that $f(y)=0$ (such a function exists since $\Omega$ is linearly convex). Since $\Omega$ is proper there exists a function $g \in \Omega^*$ such that $g(y) \neq 0$. Then 
\begin{align*}
d_{\Omega}(q_n,p) \geq \log \left(\frac{\abs{f(p)g(q_n)}}{\abs{f(q_n)g(p)}}\right)
\end{align*}
and so $d_{\Omega}(q_n,p) \rightarrow \infty$ as $n \rightarrow \infty$. Thus $B_R(p)$ is compact in $\Omega$.

Now suppose $x_n$ converges to $x$ in the metric topology, that is $d_{\Omega}(x_n,x) \rightarrow 0$. Suppose for a contradiction that $x_n$ does not converge to $x$ in the subspace topology. Then by passing to a subsequence we can suppose that $x_n$ converges to $x^* \in \overline{\Omega}$ in the subspace topology of $\overline{\Omega}$ where $x^* \neq x$. Since $d_{\Omega}(x_n,x)$ is bounded, part (1) of the lemma implies that $x^* \in \Omega$. Then $x_n$ converges to $x^*$ in the subspace topology of $\Omega$. So by the argument above $x_n$ converges to $x^*$ in the topology induced by $d_{\Omega}$. But $d_{\Omega}$ is a metric and so $x^*=x$ which contradicts our initial assumptions. 
\end{proof}

\subsection{The Hilbert metric in dimension two}\label{subsec:dim2}

Suppose $\Omega \subset \Pb(\Cb^2)$ is an open proper $\Cb$-convex set. Then there exists an affine chart $\Cb$ containing $\Omega$ and we can identify $\Pb(\Cb^2)$ with $\overline{\Cb} = \Cb \cup \{\infty\}$. In these coordinates, $\Omega$ has a well known metric $A_{\Omega}$ called the \emph{Apollonian metric}. This metric is defined by 
\begin{align*}
A_{\Omega}(z_1,z_2) = \max_{b_1,b_2 \in \overline{\Cb} \setminus \Omega} \log \frac{ \abs{z_1-b_1}\abs{z_2-b_2}}{\abs{z_2-b_1}\abs{z_1-b_2}} 
\end{align*}
where we define $\abs{z_1 -\infty}/\abs{z_2 -\infty}$ to be 1.  For information about this metric see~\cite{B1998}. The purpose of this subsection is to show that $A_\Omega = d_\Omega$ when $\Omega \subset \Pb(\Cb^2)$.

So suppose $\Omega \subset \Pb(\Cb^2)$ is an open proper $\Cb$-convex set. Then by a projective transformation, we can assume $\Omega$ is a bounded subset of the affine chart $\Cb=\{ [z:1] \in \Pb(\Cb^2)\}$. If $[f] \in \Pb(\Cb^{2*})$, then $f(z_1,z_2) = a(z_1- bz_2)$ for some $a \in \Cb^*$ and $b \in \Cb$ or $f(z_1,z_2)=bz_2$ for some $b \in \Cb^*$. In particular $[f] \in \Omega^*$ if and only if $f(z_1,z_2)=a(z_1- bz_2)$ where $a \neq 0$ and $[1:b] \notin \Omega$ or $f(z_1,z_2) = bz_2$ where $b \neq 0$. Thus if
\begin{align*}
A=\{ f(z_1,z_2)= (z_1-bz_2) : [1:b] \notin \Omega \} \cup \{ f(z_1,z_2) = z_2\} \subset (\Cb^{2})^{*}
\end{align*}
then $A$ is a set of representatives of $\Omega^*$. Then viewing $\Omega$ as a subset of $\Cb=\{[1:z]: z \in \Cb\}$, we can write $d_{\Omega}$ as
\begin{align}
\label{eq:dist_dim2}
d_{\Omega}(z_1,z_2) = \max_{b_1,b_2 \in \overline{\Cb} \setminus \Omega} \log \frac{ \abs{z_1-b_1}\abs{z_2-b_2}}{\abs{z_2-b_1}\abs{z_1-b_2}}
\end{align}
which is exactly the definition of the Apollonian metric.

One well known property of the Apollonian metric is the following:

\begin{proposition}\cite[Lemma 3.1]{B1998}\label{prop:proj_disk}
If $\Db =\{ \abs{z} <1\}$ then $d_{\Db}$ is the standard Poincar\'{e} metric on $\Db$. In particular, $(\Db,d_{\Db})$ is a geodesic metric space.
\end{proposition}

Unfortunately, as a result of Gehring and Hag demonstrates, this is essentially the only plane domains in which the Apollonian metric is geodesic.

\begin{theorem}\cite[Theorem 3.26]{GH2000}\label{thm:not_geod}
If $\Omega \subset \Cb$ is a bounded simply connected domain such that $(\Omega, d_{\Omega})$ is a geodesic metric space, then $\Omega$ is a disk.
\end{theorem}

\subsection{Hilbert metric for a projective ball}

Using Theorem~\ref{thm:dubois}, Proposition~\ref{prop:proj_disk}, and the projective ball model of complex hyperbolic $d$-space we can prove:

\begin{proposition}
Let $\Omega \subset \Pb(\Cb^{d+1})$ be a projective ball. Then $(\Omega, d_{\Omega})$ is isometric to complex hyperbolic $d$-space. 
\end{proposition}

\begin{proof}
We can pick a basis of $\Cb^{d+1}$ such that $\Omega = \{ [1:z_1: \dots : z_d] : \sum \abs{z_i}^2 < 1 \}$. Let $B_{\Omega}$ be the complex hyperbolic metric on $\Omega$ described in~\cite[Chapter 2, Section 2]{G1999}.  Suppose $L$ is the complex projective line containing $[1:0:\dots:0]$ and $[0:1:0:\dots:0]$. Let $P$ be the Poincar\'{e} metric on the disk $\Omega \cap L=\{ [1:z:0:\dots : 0] : \abs{z} < 1\}$. By Theorem~\ref{thm:dubois} and Proposition~\ref{prop:proj_disk} 
\begin{align*}
d_{\Omega}(x,y) = d_{\Omega \cap L}(x,y) = P(x,y)
\end{align*}
for all $x,y \in L \cap \Omega$. Also by the construction in~\cite[Chapter 2, Section 2]{G1999} 
\begin{align*}
B_{\Omega}(x,y) = P(x,y)
\end{align*}
for all $x,y \in L \cap \Omega$. Since $\SU(1,d)$ acts transitively on the set of complex projective lines intersecting $\Omega$ we then see that $B_{\Omega} = d_{\Omega}$ on all of $\Omega$.
\end{proof}

\subsection{The Hilbert metric is quasi-geodesic for $C^1$ domains}

In this subsection we will show that $(\Omega,d_{\Omega})$ is quasi-geodesic when $\partial \Omega$ is  $C^1$. More precisely:

\begin{theorem}
\label{thm:quasi_geod}
Suppose $\Omega \subset \Pb(\Cb^{d+1})$ is a proper linearly convex open set with $C^1$ boundary. If $\Omega$ is divisible, then $(\Omega,d_{\Omega})$ is a quasi-geodesic metric space.
\end{theorem}

Theorem~\ref{thm:quasi_geod} will follow from the next proposition, but first some notation: for a $C^1$ embedding $f: S^1 \rightarrow \Cb$ it is well known that $\textrm{Im}(f)$ bounds an open bounded set which we will denote by $\Omega_f$. 

\begin{proposition}
\label{prop:abc}
Suppose $f: S^1 \rightarrow \Cb$ is a $C^1$ embedding, then there exists $K>0$ such that $(\Omega_f, d_{\Omega_f})$ is $(1,K)$-quasi-isometric to $(\Db,d_{\Db})$. Moreover for all $\epsilon >0$ there exists $\delta >0$ such that if $g: S^1 \rightarrow \Cb$ is a $C^1$ embedding with
\begin{align*}
\max_{e^{i\theta} \in S^1} \left\{ \abs{ f(e^{i\theta})-g(e^{i\theta})} + \abs{ df(e^{i\theta})-d
g(e^{i\theta})} \right\} < \delta
\end{align*}
then $(\Omega_{g}, d_{\Omega_{g}})$ is $(1,K+\epsilon)$-quasi-isometric to $(\Db,d_{\Db})$. 
\end{proposition}

Delaying the proof of the proposition we prove Theorem~\ref{thm:quasi_geod}.

\begin{proof}[Proof of Theorem~\ref{thm:quasi_geod}]
We first claim that if $L$ is a complex projective line intersecting $\Omega$ then $L \cap \partial\Omega \subset L$ is a $C^1$ embedded submanifold. It is enough to show that $L$ intersects $\partial \Omega$ transversally at every point $x \in \partial \Omega \cap L$. Suppose not, then there exists $x \in \partial \Omega \cap L$ such that $L \subset T_x^{\Cb} \partial \Omega$. But then by $\Cb$-convexity $L \cap \Omega = \emptyset$. Thus we have a contradiction and so $L \cap \partial\Omega \subset L$ is a $C^1$ embedded submanifold. Since $\Omega$ is linearly convex and $\partial \Omega$ is $C^1$, $\Omega$ is $\Cb$-convex (see for instance~\cite[Corollary 2.5.6]{APS2004}) and thus $L \cap \partial\Omega$ is an embedded copy of $S^1$. 

Then, by Proposition~\ref{prop:abc}, for each complex projective line $L$ intersecting $\Omega$ there exists $k(L)>0$ such that $(\Omega \cap L, d_{\Omega \cap L})$ is $(1,k(L))$-quasi-isometric to $(\Db,d_{\Db})$. Moreover, for $L^\prime$ sufficiently close to $L$ Proposition~\ref{prop:abc} implies that $(\Omega \cap L^\prime, d_{\Omega \cap L^\prime})$ is $(1,k(L)+1)$-quasi-isometric to $(\Db,d_{\Db})$. 

Now let $\Gamma \leq \PSL(\Cb^{d+1})$ be a dividing group, then there exists $K \subset \Omega$ compact such that $\Omega  = \cup_{\gamma \in \Gamma} \gamma K$. The set of complex projective lines intersecting $K$ is compact and so by the remarks above there exists $k>0$ such that $(L \cap \Omega, d_{L \cap \Omega})$ is $(1,k)$-quasi-isometric to $(\Db,d_{\Db})$ for any complex projective line $L$ intersecting $K$. In particular, if $x \in K$ and $y \in \Omega$ there is a $(1,k)$-quasi geodesic joining $x$ to $y$. As $\Omega  = \cup_{\gamma \in \Gamma} \gamma K$ we then have that any two points in $\Omega$ are joined by a $(1,k)$-quasi geodesic.
\end{proof}

Proposition~\ref{prop:abc} will follow from the next three lemmas.

\begin{lemma}
\label{lem:abc1}
Suppose $\Omega_1, \Omega_2 \subset \Cb$ are open bounded sets. If $F: \overline{\Omega}_1 \rightarrow \overline{\Omega}_2$ is a $k$-bi-Lipschitz homeomorphism with $F(\Omega_1)=\Omega_2$, then $F$ induces a $(1,4\log k)$-quasi-isometry $(\Omega_1, d_{\Omega_1}) \rightarrow (\Omega_2, d_{\Omega_2})$. 
\end{lemma}

\begin{proof}
Since
\begin{align*}
\frac{1}{k} \abs{x-y} \leq \abs{f(x)-f(y)} \leq k \abs{x-y}
\end{align*}
for all $x,y \in \overline{\Omega}_1$ we have, from Equation~(\ref{eq:dist_dim2}), that 
\begin{align*}
d_{\Omega_1}(x,y) - 4\log k \leq d_{\Omega_2}(f(x),f(y)) \leq d_{\Omega_1}(x,y) + 4\log k
\end{align*}
for all $x,y \in \Omega_1$.
\end{proof}

\begin{lemma}
\label{lem:abc2}
Suppose $f: S^1 \rightarrow \Cb$ is a $C^1$ embedding, then for all $\epsilon >0$ there exists $\delta >0$ such that if $g: S^1 \rightarrow \Cb$ is a $C^1$ embedding with
\begin{align*}
\max_{e^{i\theta} \in S^1} \left\{ \abs{ f(e^{i\theta})-g(e^{i\theta})} + \abs{ df(e^{i\theta})-d
g(e^{i\theta})} \right\} < \delta
\end{align*}
then there exists $F: \overline{\Omega}_f \rightarrow \overline{\Omega}_g$ a $(1+\epsilon)$-bi-Lipschitz homeomorphism with $F(\Omega_f) = \Omega_g$.
\end{lemma}

\begin{proof}
Since $f:S^1 \rightarrow \Cb$ is a $C^1$ embedding there exists a collar neighborhood extension $\Phi: \{ 1 -\eta \leq \abs{z} \leq 1+\eta\} \rightarrow \Cb$. Then if $\delta$ is small enough, $\textrm{Im}(g)$ can be parameterized by $e^{i\theta} \rightarrow \Phi(r(e^{i\theta}) e^{i\theta})$ for some $C^1$ function $r:S^1 \rightarrow (1-\eta,1+\eta)$. By further shrinking $\delta$, it is easy to construct a $C^1$ diffeomorphism $F:\overline{\Omega}_f \rightarrow \overline{\Omega}_g$ such that $\abs{F^\prime(z)}$ and $\abs{(F^{-1})^\prime(z)}$ are bounded by $(1+\epsilon)$. Thus $F:\overline{\Omega}_f \rightarrow \overline{\Omega}_g$ is $(1+\epsilon)$-bi-Lipschitz.
\end{proof}

\begin{lemma}
\label{lem:abc3}
Suppose $f: S^1 \rightarrow \Cb$ is a $C^\infty$ embedding, then there exists a $k$-bi-Lipschitz homeomorphism $F: \overline{\Db} \rightarrow \overline{\Omega}_f$ with $F(\Db) = \Omega_f$.
\end{lemma}

\begin{proof} This (and more) follows from the smooth version of the Riemann mapping theorem (see for instance~\cite[Chapter 5, Theorem 4.1]{T2011}). \end{proof}

We can now prove Proposition~\ref{prop:abc}.

\begin{proof}[Proof of Proposition~\ref{prop:abc}]
Suppose $f: S^1 \rightarrow \Cb$ is a $C^1$ embedding. Then since any $C^1$ embedding can be approximated by a $C^\infty$ embedding, Lemma~\ref{lem:abc2} and Lemma~\ref{lem:abc3} implies the existence of a $k$-bi-Lipschitz map $F: \overline{\Db} \rightarrow \overline{\Omega}_f$. By Lemma~\ref{lem:abc1} this induces a $(1,4\log(k))$-quasi-isometry $(\Db,d_{\Db}) \rightarrow (\Omega_f,d_{\Omega_f})$. Finally the ``moreover'' part of the proposition is just Lemma~\ref{lem:abc1} and Lemma~\ref{lem:abc2}.
\end{proof}

\subsection{Action of the automorphism group:} 

In this subsection we establish some properties of $\Aut(\Omega)$. 

\begin{proposition}
\label{prop:proper}
Suppose $\Omega \subset \Pb(V)$ is a proper linearly convex open set, then $\Aut(\Omega)$ acts properly on $\Omega$.
\end{proposition}

\begin{remark} There are two natural topologies on $\Aut(\Omega)$: the first comes from the inclusion $\Aut(\Omega) \hookrightarrow \PSL(V)$ and the second comes from the inclusion $\Aut(\Omega) \hookrightarrow \textrm{Homeo}(\Omega)$ where $\textrm{Homeo}(\Omega)$ has the compact-open topology. The proof of this proposition shows that the two topologies coincide. 
\end{remark}

\begin{proof}
It is enough to show that the set $\{ \varphi \in \Aut(\Omega) : \varphi K \cap K \neq \emptyset\}$ is compact for any $K \subset \Omega$ compact. So assume $\{\varphi_n \}_{n \in \Nb} \subset \{ \varphi \in \Aut(\Omega) : \varphi K \cap K \neq \emptyset\}$ for some compact $K$. We claim that $\varphi_n$ has a convergent subsequence in $\Aut(\Omega)$. Let $f_n : \Omega \rightarrow \Omega$ be the homeomorphism induced by $\varphi_n$, that is $f_n(p)=\varphi_n(p)$. Since each $f_n$ is an isometry with respect to $d_{\Omega}$ and $f_n(K) \cap K \neq \emptyset$, using the Arzel\'{a}-Ascoli theorem we may pass to a subsequence such that $f_n$ converges uniformly on compact subsets of $\Omega$ to a continuous map $f: \Omega \rightarrow \Omega$. Moreover $f$ is an isometry and hence injective.

Now we can pick $\hat{\varphi}_n \in \GL(V)$ representing $\varphi_n \in \PSL(V)$ such that $\norm{\hat{\varphi}_n}=1$. Then by passing to a subsequence we can suppose $\hat{\varphi}_n \rightarrow \hat{\varphi} \in \End(V)$. By construction, for $y \in \Omega \setminus (\Omega \cap  \ker \hat{\varphi})$ we have that $\hat{\varphi}(y) = f(y)$. As $f$ is injective this implies that $\hat{\varphi}$ has full rank. Thus $[\hat{\varphi}] \in \PSL(V)$ and  $\varphi_n \rightarrow [\hat{\varphi}]$ in $\PSL(V)$. As $\Aut(\Omega)$ is closed in $\PSL(V)$ this implies that  $\varphi_n \rightarrow [\hat{\varphi}]$ in $\Aut(\Omega)$.

Since $\varphi_n$ was an arbitrary sequence in $ \{ \varphi \in \Aut(\Omega) : \varphi K \cap K \neq \emptyset\}$ this implies that $ \{ \varphi \in \Aut(\Omega) : \varphi K \cap K \neq \emptyset\}$ is compact. Since $K$ was an arbitrary compact set of $\Omega$ the proposition follows. 
\end{proof}

Now suppose $\Gamma \leq \PSL(V)$ is a discrete group dividing a proper $\Cb$-convex set $\Omega$ with $C^1$ boundary. By Theorem~\ref{thm:quasi_geod} $(\Omega,d_{\Omega})$ is a quasi-geodesic metric space and by Proposition~\ref{prop:proper} $\Gamma$ acts properly on $(\Omega,d_{\Omega})$. Then by Theorem~\ref{thm:fundgeomgp}, $\Gamma$ is finitely generated and so by applying Selberg's Lemma we obtain:

\begin{corollary}
\label{cor:fin_gen}
Suppose $\Omega$ is a proper linearly convex open set with $C^1$ boundary. If $\Omega$ is divisible, then $\Omega$ is divisible by some torsion free subgroup of $\PSL(\Cb^{d+1})$.
\end{corollary}

For torsion free dividing groups we have the following:

\begin{corollary}
\label{cor:inj_rad}
Suppose $\Omega$ is a proper linearly convex open set and $\Gamma \leq \Aut(\Omega)$ is a torsion free discrete group dividing $\Omega$. Then there exists $\epsilon >0$ such that $d_{\Omega}(\gamma p, p) > \epsilon$ for all $\gamma\in \Gamma \setminus \{1\}$ and for all $p \in \Omega$.
\end{corollary}

\begin{proof}
Since $\Gamma$ is torsion free, discrete, and acts properly
\begin{align*}
\inf_{\gamma \in \Gamma \setminus \{1\} } d_{\Omega}(\gamma p, p) >0
\end{align*}
for all $p \in \Omega$. Now since $\Gamma$ divides $\Omega$, there exists $K \subset \Omega$ compact such that $\Omega = \cup_{\gamma \in \Gamma} \gamma K$. Then
\begin{align*}
\inf_{p \in \Omega} \inf_{\gamma \in \Gamma \setminus \{1\}} d_{\Omega}(\gamma p,p) = \inf_{p \in K} \inf_{\gamma \in \Gamma \setminus \{1\}} d_{\Omega}(\gamma p,p) >0.
\end{align*}
\end{proof}

\subsection{Boundary behavior} As mentioned in the introduction one nice feature of the Hilbert metric is that the behavior of the metric near the boundary is closely related to the geometry of the boundary.

\begin{proposition}
\label{prop:bd_behav}
Suppose $\Omega$ is a proper linearly convex open set. If $\{p_n\}_{n \in \Nb},\{ q_n\}_{n \in \Nb} \subset \Omega$ are sequences  such that $p_n \rightarrow x \in \partial \Omega$, $q_n \rightarrow y \in \partial \Omega$, and $d_{\Omega}(p_n,q_n) < R$ for some $R>0$ then every complex tangent hyperplane of $\Omega$ containing $x$ also contains $y$.
\end{proposition}

\begin{proof}
Since $\Omega$ is proper there exists $g \in \Omega^*$ such that $g(x) \neq 0$ and $g(y) \neq 0$. If $H$ is a complex tangent hyperplane containing $x$ and $f \in\Pb(V^*)$ is such that $[\ker f]=H$, then $[\ker f] \cap \Omega = H \cap \Omega = \emptyset$. Thus $f \in \Omega^*$ and 
\begin{align*}
R \geq d_{\Omega}( p_n, q_n ) \geq \log \abs{ \frac{f(q_n)}{f(p_n)}}+\log  \abs{\frac{g(p_n)}{g(q_n)}}.
\end{align*}
Let $\hat{p}_n,\hat{q}_n,\hat{x},\hat{y} \in \Cb^{d+1}$ and $\hat{f},\hat{g} \in \Cb^{(d+1)*}$ be representatives of $p_n,q_n,x,y \in \Pb(\Cb^{d+1})$ and $f,g \in \Pb(\Cb^{(d+1)*})$ normalized such that 
\begin{align*}
\norm{\hat{f}}=\norm{\hat{g}}=\norm{\hat{p}_n}=\norm{\hat{q}_n}=\norm{\hat{x}}=\norm{\hat{y}}=1.
\end{align*}
Then 
\begin{align*}
R \geq \log \abs{\frac{\hat{f}(\hat{q}_n)}{\hat{f}(\hat{p}_n)}}+\log \abs{ \frac{\hat{g}(\hat{p}_n)}{\hat{g}(\hat{q}_n)}}.
\end{align*}
Since $f(x)=0$, we see that $\hat{f}(\hat{p}_n) \rightarrow 0$. Since $g(x) \neq 0$ and $g(y) \neq 0$, we see that 
\begin{align*}
\log  \abs{\frac{\hat{g}(\hat{p}_n)}{\hat{g}(\hat{q}_n)}}
\end{align*}
 is bounded from above and below (for $n$ large). Thus we must have that $\hat{f}(\hat{q}_n) \rightarrow 0$ and then we see that $y \in [\ker f]$.
\end{proof}

\subsection{Translation distance} As mentioned in the introduction one nice feature of the Hilbert metric is that it is possible to estimate $d_{\Omega} (\varphi y, y)$ for $\varphi \in \Aut(\Omega)$. 

\begin{proposition}
\label{prop:trans_dist}
Suppose $\Omega$ is a  proper linearly convex open set. If $x_0 \in \Omega$ then there exist $R>0$ depending only on $x_0$ such that 
\begin{align*}
d_{\Omega}(\varphi x_0, x_0) \leq R + \log\left(\norm{\varphi}\norm{\varphi^{-1}}\right)
\end{align*}
for all $\varphi \in \Aut(\Omega)$.
\end{proposition}

\begin{proof}
Let 
\begin{align*}
\Lambda = \{ f \in \Cb^{(d+1)*} : \norm{f}=1, [f] \in \Omega^*\}.
\end{align*}
Since $\Omega^*$ is $\Aut(\Omega)$-invariant we see that 
\begin{align*}
^t\varphi f/\norm{^t\varphi f} \in \Lambda
\end{align*}
whenever $f \in \Lambda$ and $\varphi \in \Aut(\Omega)$ . Let $\hat{x}_0 \in \Cb^{d+1}$ as a representative of $x_0 \in \Pb(\Cb^{d+1})$ with norm one. Since $f(x_0) \neq 0$ for all $f \in \Lambda$ and $\Lambda$ is compact,  there exists $C >0$ such that: 
\begin{align*}
-C < \log \abs{f(\hat{x}_0)} < C
\end{align*}
 for all $f \in \Lambda$. Now for $\varphi \in \Aut(\Omega)$
 \begin{align*}
 d_{\Omega}(\varphi x_0, x_0) 
 = \sup_{f,g \in \Lambda} \log \abs{\frac{f(\varphi \hat{x}_0) g(\hat{x}_0)}{f(\hat{x}_0)g(\varphi \hat{x}_0)}} 
  \leq 2C + \sup_{f,g \in \Lambda} \log \abs{\frac{f(\varphi \hat{x}_0)}{g(\varphi \hat{x}_0)}}
 \end{align*}
 and for $f,g \in \Lambda$
 \begin{align*}
\log \abs{\frac{f(\varphi \hat{x}_0)}{g(\varphi \hat{x}_0)}}
&=\log \frac{\norm{ ^t\varphi f}}{\norm{^t\varphi g}}+\log \abs{ \left(\frac{^t\varphi f}{\norm{^t\varphi f}}\right)(\hat{x}_0)}-\log \abs{ \left(\frac{^t\varphi g}{\norm{^t\varphi g}}\right)(\hat{x}_0)} \\
& \leq \log\left(\norm{\varphi}\norm{\varphi^{-1}}\right)+\sup_{f^\prime,g^\prime \in \Lambda} \log \abs{\frac{f^\prime(\hat{x}_0)}{g^\prime(\hat{x}_0)}} \\
& \leq \log\left(\norm{\varphi}\norm{\varphi^{-1}}\right)+2C.
 \end{align*}
Thus 
\begin{align*}
d_{\Omega}(\varphi x_0, x_0) \leq 4C + \log\left(\norm{\varphi}\norm{\varphi^{-1}}\right)
\end{align*}
and the proposition holds with $R:=4C$.
 \end{proof}

\section{Every element is bi-proximal or almost unipotent}\label{sec:bi_prox_or_unip}

For $V$ a complex $(d+1)$-dimensional vector space and $\varphi \in \PSL(V)$ let
\begin{align*}
\sigma_1(\varphi) \leq \sigma_2(\varphi) \leq \dots \leq \sigma_{d+1}(\varphi)
\end{align*}
be the absolute value of the eigenvalues (counted with multiplicity) of $\varphi$. Since we are considering absolute values this is well defined. 

\begin{definition} \
\begin{enumerate}
\item An element $\varphi \in \PSL(V)$ is called \emph{proximal} if $\sigma_{d}(\varphi) < \sigma_{d+1}(\varphi)$ and is called \emph{bi-proximal} if $\varphi$ and $\varphi^{-1}$ are proximal. 
\item An element $\varphi \in \PSL(V)$ is called \emph{almost unipotent} if 
\begin{align*}
\sigma_1(\varphi) = \sigma_2(\varphi) = \dots = \sigma_{d+1}(\varphi)=1.
\end{align*}
\end{enumerate}
\end{definition}

When $\varphi$ is bi-proximal let $x^+_{\varphi}$ and $x^-_{\varphi}$ be the eigenlines in $\Pb(V)$ corresponding to $\sigma_{d+1}(\varphi)$ and $\sigma_1(\varphi)$.  
The purpose of this section is to prove the following.

\begin{theorem}
\label{thm:bi_prox}
Suppose $\Omega$ is a proper $\Cb$-convex open set with $C^1$ boundary. If $\Gamma \leq \PSL(\Cb^{d+1})$ divides $\Omega$ then every $\gamma \in \Gamma \setminus \{1\}$ is bi-proximal or almost unipotent. Moreover if $\varphi \in \Aut(\Omega)$ is bi-proximal then 
\begin{enumerate}
\item $x_{\varphi}^+,x_{\varphi}^- \in \partial \Omega$, 
\item $T^{\Cb}_{x_{\varphi}^+} \partial \Omega \cap \partial \Omega = \{ x_{\varphi}^+\}$,
\item $T^{\Cb}_{x_{\varphi}^-} \partial \Omega \cap \partial \Omega = \{ x_{\varphi}^-\}$, and
\item if $U^+ \subset \overline{\Omega}$ is a neighborhood of $x^+_{\varphi}$ and $U^- \subset \overline{ \Omega}$ is a neighborhood of $x^-_{\varphi}$ then there exists $N>0$ such that for all $m > N$ we have
\begin{align*}
\varphi^m(\partial \Omega \setminus U^-) \subset U^+ \text{ and } \varphi^{-m}(\partial \Omega \setminus U^+) \subset U^-.
\end{align*}
\end{enumerate}
\end{theorem}

\begin{remark} Notice that in the second part of theorem we allow $\varphi$ to be any bi-proximal element in $\Aut(\Omega)$. 
\end{remark}

Given an element $\varphi \in \SL(V)$ let $m^+(\varphi)$ be the size of the largest Jordan block of $\varphi$ whose corresponding eigenvalue has absolute value $\sigma_{d+1}(\varphi)$. Next let $E^+(\varphi)$ be the span of the eigenvectors of $\varphi$ whose eigenvalue have absolute value $\sigma_{d+1}(\varphi)$ and are part of a Jordan block with size $m^+(\varphi)$. Also define $E^-(\varphi) = E^+(\varphi^{-1})$.

Given $y \in \Pb(V)$ let $L(\varphi,y) \subset \Pb(V)$ denote the limit points of the sequence $\{\varphi^n y\}_{n \in \Nb}$. With this notation we have the following observations:

\begin{proposition}
\label{prop:attracting}
Suppose $\varphi \in \SL(V)$ and $\{ \varphi^n\}_{ n \in \Nb} \subset \SL(V)$ is unbounded, then
\begin{enumerate}
\item there exists a proper projective subspace $H \subsetneq \Pb(V)$ such that $L(\varphi,y) \subset [E^+(\varphi)]$ for all $y \in \Pb(V) \setminus H$,
\item $\varphi$ acts recurrently on $[E^+(\varphi)] \subset \Pb(V)$, that is for all $y \in [E^+(\varphi)]$ there exists $n_k \rightarrow \infty$ such that $\varphi^{n_k}y \rightarrow y$,
\item $E^-(\varphi) \subset \ker f$ for all $f \in E^+(^t\varphi)$.
\end{enumerate}
\end{proposition}

\begin{proof} All three statements follow easily once $\varphi$ is written in Jordan normal form.
\end{proof}

\begin{lemma}
\label{lem:weak_attract}
Suppose $\Omega$ is a proper $\Cb$-convex open set with $C^1$ boundary and $\varphi \in \Aut(\Omega)$ such that $\{\varphi^n \}_{n \in \Nb} \subset \PSL(\Cb^{d+1})$ is unbounded. Then $E^\pm(\varphi) = x^\pm$ for some $x^\pm \in \partial \Omega$ and $E^\pm(^t\varphi) = f^\pm$ for some $f^\pm\in \Omega^*$. Moreover $T_{x^\pm}^{\Cb} \partial \Omega = [\ker f^\mp]$.
\end{lemma}

\begin{proof} We will break the proof of the lemma into a series of claims. \vspace*{5pt}

\noindent \textbf{Claim 1:} \textit{$[E^+(^t\varphi)] \cap \Omega^*$ is non-empty.} \vspace*{5pt}

By part (1) of Proposition~\ref{prop:attracting} there exists a hyperplane $H \subset \Pb(V^*)$  such that $L(^t\varphi, f) \subset [E^+(^t\varphi)]$ for all $f \in \Pb(V^*)\setminus H$. By Proposition~\ref{prop:zar_den}, $\Omega^*$ is not contained in a hyperplane and so there exists $f \in \Omega^* \setminus \Omega^* \cap H$. Then as $\Omega^*$ is compact and $^t\varphi$-invariant, $L(^t\varphi, f) \subset \Omega^*$ and thus $[E^+(^t\varphi)] \cap \Omega^* \neq \emptyset$. \vspace*{5pt}

\noindent \textbf{Claim 2:} \textit{$[E^-(\varphi)] \cap \Omega = \emptyset$ and $[E^-(\varphi)] \cap \partial \Omega \neq \emptyset$. In particular, since $\partial \Omega$ is $C^1$ if $x \in[E^-(\varphi)] \cap \partial\Omega$ then $[E^-(\varphi)] \subset T_x^{\Cb} \partial \Omega$. }\vspace*{5pt}

By Proposition~\ref{prop:proper}, $\Aut(\Omega)$ acts properly on $\Omega$ and hence for any $y \in \Omega$ the set
\begin{align*}
\{ n \in \Nb : d_{\Omega}(\varphi^{-n} y,y) \leq 1\}
\end{align*}
is finite. So $[E^-(\varphi)] \cap \Omega = \emptyset$ by part (2) of Proposition~\ref{prop:attracting}. Since $\Omega$ is open,  part (1) of Proposition~\ref{prop:attracting} implies the existence of some  $y \in \Omega$ such that  $L(\varphi^{-1}, y) \subset [E^-(\varphi)]$.  Since $\Omega$ in $\varphi$-invariant $L(\varphi^{-1},y) \subset \overline{\Omega}$. Thus  $[E^-(\varphi)] \cap \overline{\Omega} \neq \emptyset$. \vspace*{5pt}

\noindent \textbf{Claim 3:}  \textit{$\{ f^+ \} = [E^+(^t\varphi)] \cap \Omega^*$ for some $f^+ \in \Pb(V^*)$ and $[\ker f^+] = T_x^{\Cb} \partial \Omega$ for any $x \in [E^-(\varphi)] \cap \partial \Omega$. }\vspace*{5pt}

Suppose $f \in [E^+(^t\varphi)] \cap \Omega^*$ then by part (3) of Proposition~\ref{prop:attracting}, $E^-(\varphi) \subset \ker f$ and by the definition of $\Omega^*$, $[\ker f] \cap \Omega = \emptyset$. Thus if $x \in[E^-(\varphi)] \cap \partial \Omega$ then $[\ker f]$ is a complex tangent hyperplane of $\Omega$ at $x$. Since $\partial \Omega$ is $C^1$ this implies that $[\ker f] = T_x^{\Cb} \partial \Omega$.  As $f \in [E^+(^t\varphi)] \cap \Omega^*$ was arbitrary this implies the claim.\vspace*{5pt}

\noindent \textbf{Claim 4:} \textit{$f^+ = E^+(^t\varphi)$ for some $f^+ \in \Omega^*$.}\vspace*{5pt}

Pick representatives $\hat{\varphi}_n \in \GL(V^*)$ of $^t\varphi^{n} \in \PSL(V^*)$ such that $\norm{\hat{\varphi}_n}=1$. Then there exists $n_k \rightarrow \infty$ such that $\hat{\varphi}_{n_k}$ converges to a linear endomorphism $\hat{\varphi}_\infty \in \End(V^*)$.  By construction $\hat{\varphi}_{\infty}(g) \in L(^t\varphi,g)$ for any $g \in \Pb(V^*) \setminus [\ker \hat{\varphi}_{\infty}]$. Also by using the Jordan normal form one can check that $\hat{\varphi}_\infty(V^*)=E^+(^t\varphi)$. Select $f \in \Pb(V^*)$ such that $\hat{\varphi}_{\infty}(f)=f^+$. Then viewing $f$ and $f^+$ as complex one dimensional subspaces of $V^*$ we see that 
\begin{align*}
W:=\{ v \in V^* : \hat{\varphi}_{\infty}(v) \in f^+\} = f + \ker \hat{\varphi}_{\infty}.
\end{align*}
Notice that 
\begin{align*}
\dimension_{\Cb} W = 1 + \dimension_{\Cb} \ker \hat{\varphi}_{\infty} = d+2-\dimension_{\Cb} E^+(^t\varphi).
\end{align*}
Finally assume for a contradiction that $\dimension_{\Cb} E^+(^t\varphi) > 1$. In this case $[W]$ is a proper projective subspace of $\Pb(V^*)$. By Proposition~\ref{prop:zar_den}, $\Omega^*$ is not contained in a hyperplane and thus there exists $g \in \Omega^* \setminus \Omega^* \cap [W]$. Since $\ker \hat{\varphi}_{\infty} \subset W$, $g \notin [\ker \hat{\varphi}_{\infty}]$ and so $\hat{\varphi}_{\infty}g$ is well defined in $\Pb(V^*)$. As $\hat{\varphi}_{\infty}(V^*) = E^+(^t\varphi)$ we then have that $\hat{\varphi}_{\infty}g \in [E^+(^t\varphi)]$. Since $\Omega^*$ is compact and $^t\varphi$-invariant, $\hat{\varphi}_{\infty}g \in L(^t\varphi,g) \subset \Omega^*$.
Thus by Claim 3 we must have that $\hat{\varphi}_{\infty}g = f^+$. But this contradicts the fact that $g \notin [W]$. So we have a contradiction and so $E^+(^t\varphi)$ must be a one complex dimensional subspace.\vspace*{5pt}

\noindent \textbf{Claim 5:} \textit{$x^+ = E^+(\varphi)$ for some $x^+ \in \partial\Omega$.}\vspace*{5pt}

The property of $E^+(\varphi)$ having dimension one depends only on the Jordan block structure of $\varphi$. As $\varphi$ and $^t\varphi$ have the same Jordan block structure Claim 4 implies that $E^+(\varphi) = x^+$ for some $x^+ \in \Pb(\Cb^{d+1})$. By repeating the argument in the proof of Claim 2 we see that $x^+ \in \partial \Omega$. 

\vspace*{5pt} \noindent \textbf{Claim 6:} \textit{Lemma~\ref{lem:weak_attract} is true. }\vspace*{5pt}

Summarizing our conclusions so far: we have that $E^+(\varphi)=x^+$ for some $x^+\in \partial \Omega$, $E^+(^t\varphi)=f^+$ for some $f^+ \in \Omega^*$, and $[\ker f^+]$ is a complex tangent hyperplane of $\Omega$ containing $[E^-(\varphi)]$. Thus applying the above argument to $\varphi^{-1}$ we see that $E^-(\varphi)=x^-$ for some $x^- \in \partial \Omega$, $E^-(^t\varphi)=f^-$ for some $f^- \in \Omega^*$, and $[\ker f^-]$ is a complex tangent hyperplane of $\Omega$ containing $[E^+(\varphi)]$. Since $E^\pm(\varphi) = x^\pm$ we see that $[\ker f^\mp]$ is a complex tangent hyperplane containing $x^\pm$ and thus $T_{x^\pm}^{\Cb} \partial \Omega = [\ker f^{\mp}]$.
\end{proof}

Since each $\varphi \in \SL(V)$ is either almost unipotent or $\sigma_{d+1}(\varphi) > \sigma_1(\varphi)$, Theorem~\ref{thm:bi_prox} will follow from the next lemma.

\begin{lemma}
Suppose $\Omega$ is a proper $\Cb$-convex open set with $C^1$ boundary and $\varphi \in \Aut(\Omega)$ is such that $\sigma_{d+1}(\varphi ) > \sigma_1(\varphi )$. Then $\varphi $ is bi-proximal and
\begin{enumerate}
\item $T^{\Cb}_{x_{\varphi }^+} \partial \Omega \cap \partial \Omega = \{ x_{\varphi }^+\}$, 
\item $T^{\Cb}_{x_{\varphi }^-} \partial \Omega \cap \partial \Omega = \{ x_{\varphi }^-\}$,
\item if $U^+ \subset \overline{\Omega}$ is a neighborhood of $x^+_{\varphi }$ and $U^- \subset \overline{\Omega}$ is a neighborhood of $x^-_{\varphi }$ then there exists $N>0$ such that for all $m > N$ we have
\begin{align*}
\varphi ^m(\partial \Omega \setminus U^-) \subset U^+ \text{ and } \varphi ^{-m}(\partial \Omega \setminus U^+) \subset U^-.
\end{align*}
\end{enumerate}
\end{lemma}

\begin{proof}
Let $E^{\pm}(^t\varphi) = f^\pm \in \Omega^*$  and $E^\pm(\varphi)=x^\pm \in \partial \Omega$. By the previous lemma $[\ker f^{\pm} ]= T^{\Cb}_{x^\mp} \partial \Omega$. Since $\sigma_{d+1}(\varphi ) > \sigma_1(\varphi )$ we see that $E^+(^t\varphi) \neq E^-(^t\varphi)$ and so $f^+ \neq f^-$. Since $\partial \Omega$ is $C^1$, $x^+$ is contained in unique complex tangent hyperplane and so $x^+ \notin [\ker f^+]$. For the same reason $x^- \notin [\ker f^-]$. 

Then there exists a basis $e_1, e_2, \dots, e_{d+1}$ of $\Cb^{d+1}$ such that $\Cb\cdot e_1 = x^+$, $\Cb \cdot e_2 = x^-$, and $\ker f^+ \cap \ker f^-$ is the span of $e_3,\dots, e_{d+1}$. Since $x^+$, $x^-$, and $\ker f^+ \cap \ker f^-$ are $\varphi$-invariant, with respect to this basis $\varphi$ is represented by a matrix of the form
\begin{align*}
\begin{pmatrix}
\lambda^+ & 0 & 0 \\
0 & \lambda^- & 0 \\
0 & 0 & A \\
\end{pmatrix} \in \SL(\Cb^{d+1})
\end{align*}
where $A$ is some $(d-1)$-by-$(d-1)$ matrix. Finally since $E^+(\varphi) = x^+$ and $E^-(\varphi)=x^-$ we see that $\varphi$ is bi-proximal. 

We now show part (1) of the lemma, that is $[\ker f^-] \cap \partial \Omega = \{ x^+\}$. If $x \in [\ker f^-] \cap \partial \Omega$ then with respect to the basis above $x = [ w_1 : 0 : w_2 : \dots : w_d]$ for some $w_1,\dots,w_d \in \Cb$. If $w_1=0$ then $x \in [\ker f^+]$. But then $[\ker f^+]$ and $[\ker f^-]$ are complex tangent hyperplanes to $\partial \Omega$ at $x$. Since $\partial \Omega$ is $C^1$ this implies that $[\ker f^+]=T_x^{\Cb}\partial\Omega =[\ker f^-]$ which contradicts the fact that $f^+$ and $f^-$ are distinct points in $\Pb(V^*)$. So $w_1 \neq 0$, but then either $x=x^+$ or any limit point of $\{\varphi^{-n} x\}_{n \in \Nb}$ is in $[\ker f^+] \cap [\ker f^-] \cap \partial \Omega$ which we just showed is empty. So $[\ker f^-] \cap \partial \Omega = \{ x^+\}$. A similar argument shows that $[\ker f^+] \cap \partial \Omega = \{ x^-\}$.

By part (2) of the lemma, with respect to the basis above 
\begin{align*}
\overline{\Omega} \setminus \{ x^-\} \subset \{ [1:z_1:\dots z_d] : z_1,\dots,z_d \in \Cb\}.
\end{align*}
Thus for all $x \in \overline{\Omega}$ either $x=x^-$ or $\varphi^m x \rightarrow x^+$ as $m \rightarrow \infty$. In a similar fashion, for all $x \in \overline{\Omega}$ either $x=x^+$ or $\varphi^m x \rightarrow x^-$ as $m \rightarrow -\infty$. Thus part (3) holds. 
\end{proof}

\section{$\Gamma$ contains a bi-proximal element}\label{sec:exist_bi}

The purpose of this section is to prove the following.

\begin{theorem}
\label{thm:exist_bi}
Suppose $\Omega$ is a proper $\Cb$-convex open set with $C^1$ boundary. If $\Gamma \leq \PSL(\Cb^{d+1})$ divides $\Omega$ then some $\gamma \in \Gamma \setminus \{1\}$  is bi-proximal.
\end{theorem}

\begin{remark}
Using Proposition~\ref{prop:trans_dist}, if $\varphi \in \Aut(\Omega)$ is almost unipotent and $x_0 \in \Omega$ then we have an estimate of the form:
\begin{align}
\label{eq:abcd}
d_{\Omega}(\varphi^N x_0,x_0) \leq R + \log \left(\norm{\varphi^N}\norm{\varphi^{-N}}\right) \leq A+B\log(N).
\end{align}
In particular, if we knew that every non-trivial element of $\Gamma$ is an ``axial isometry'' then we would immediately deduce that every non-trivial element of $\Gamma$ is bi-proximal. Unfortunately, we do not see a direct way of establishing that every non-trivial element of $\Gamma$ is an ``axial isometry.'' \end{remark}

We will start with a definition, but first let $\SL^*(\Cb^{d+1}) = \{ \varphi \in \GL(\Cb^{d+1}) : \abs{\det \varphi }=1\}$.

\begin{definition}
A connected closed Lie subgroup $G \leq \SL^*(\Cb^{d+1})$ is called \emph{almost unipotent} if there exists a flag 
\begin{align*}
\{0\} = V_0 \subsetneq V_1 \subsetneq  \dots \subsetneq V_k \subsetneq V_{k+1} = \Cb^{d+1}
\end{align*}
preserved by $G$ such that if $G_{i+1} \leq \GL(V_{i+1}/V_i)$ is the projection of $G$ then $G_{i+1}$ is bounded. 
\end{definition} 

\begin{remark}
Notice that a group $G$ is unipotent if and only if there exists a complete flag 
\begin{align*}
\{0\} = V_0 \subsetneq V_1 \subsetneq  \dots \subsetneq V_{d} \subsetneq V_{d+1} = \Cb^{d+1}
\end{align*}
preserved by $G$ such that if $G_{i+1} \leq \GL(V_{i+1}/V_i)$ is the projection of $G$ then $G_{i+1}=\{1\}$.
\end{remark}

The proof of Theorem~\ref{thm:exist_bi} will use the next two propositions.

\begin{proposition}
\label{prop:Ugp}
Suppose $\Gamma \leq \SL(\Cb^{d+1})$ is a subgroup such that every $\gamma \in \Gamma$ is almost unipotent. Let $G$ be the Zariski closure of $\Gamma$ in $\SL(\Rb^{2d+2})$. If $G$ is connected, then $G$ is almost unipotent.
\end{proposition}

\begin{proposition}
\label{prop:growth}
Suppose $G \leq \SL(\Cb^{d+1})$ is a connected closed Lie subgroup. If $G$ is almost unipotent and $g_1, \dots, g_k$ are fixed elements in $G$ then there exists a constant $C>0$ such that for all $N>0$ and $i_1, \dots, i_N \in \{1,\dots,k\}$
\begin{align*}
\norm{g_{i_1} \cdots g_{i_N}} \leq CN^d.
\end{align*}
\end{proposition}

Delaying the proof of the propositions we will prove Theorem~\ref{thm:exist_bi}.

\begin{proof}[Proof of Theorem~\ref{thm:exist_bi}]
Suppose for a contradiction that $\Gamma$ contains no bi-proximal elements, then by Theorem~\ref{thm:bi_prox} every element of $\Gamma$ is almost unipotent. If $\pi:\SL(\Cb^{d+1})\rightarrow \PSL(\Cb^{d+1})$ is the natural projection, then there exists a finite index subgroup $\Gamma^{\prime} \leq \pi^{-1}(\Gamma)$ such that the Zariski closure of $\Gamma^{\prime}$ is connected and $\Gamma^\prime$ is torsion free. Since $\Gamma^\prime$ is torsion free $\pi$ induces an isomorphism $\Gamma^\prime \rightarrow \pi(\Gamma^\prime)$ and by construction $\pi(\Gamma^\prime) \leq \Gamma$ will have finite index. Then $\pi(\Gamma^\prime)$ divides $\Omega$ and hence by Theorem~\ref{thm:quasi_geod}, Proposition~\ref{prop:proper}, and Theorem~\ref{thm:fundgeomgp} $\Gamma^\prime$ is finitely generated. Now fix a finite generating set $S=\{s_1,\dots,s_k\} \subset \Gamma^\prime$ and a point $x_0 \in \Omega$. By Proposition~\ref{thm:fundgeomgp} there exists $A,B >0$ such that the map $\gamma \in \Gamma^\prime \rightarrow \gamma \cdot x_0$ is an $(A,B)$-quasi-isometry between $(\Gamma^\prime, d_S)$ and $(\Omega,d_{\Omega})$. 

Since $\Gamma^\prime$ has infinite order, there exists a sequence $i_1, i_2, \dots \in \{1,\dots, k\}$ such that the map 
\begin{align*}
N \in \Nb \rightarrow \gamma_{i_N} \cdots \gamma_{i_2}\gamma_{i_1}
\end{align*}
is a geodesic with respect to the word metric, that is 
\begin{align*}
d_S(\gamma_{i_N} \gamma_{i_{N-1}} \cdots \gamma_{i_1}, 1) = N
\end{align*}
for all $N>0$. Then 
\begin{align}
\label{eq:est}
\frac{1}{A}N-B \leq d_{\Omega}(\gamma_{i_N} \gamma_{i_{N-1}} \cdots \gamma_{i_1} x_0, x_0) \leq AN+B
\end{align}
for all $N >0$. 

Since the Zariski closure of $\Gamma^\prime$ is connected, Proposition~\ref{prop:Ugp} and Proposition~\ref{prop:growth} implies the existence of $C>0$  such that 
\begin{align*}
\norm{\gamma_{i_1} \cdots \gamma_{i_N}} \leq CN^d \text{ and } \norm{\gamma_{i_N}^{-1} \cdots \gamma_{i_1}^{-1}} \leq CN^d
\end{align*}
for all $N>0$. 

Then Proposition~\ref{prop:trans_dist} implies that:
\begin{align*}
d_{\Omega}(\gamma_{i_N} \cdots \gamma_{i_1} x_0, x_0) \leq R + \log\left(\norm{\gamma_{i_N} \cdots \gamma_{i_1} }\norm{\gamma_{i_1}^{-1} \cdots \gamma_{i_N}^{-1} }\right) \leq  R + 2\log CN^d
\end{align*}
for some $R >0$ depending only on $x_0$. This contradicts the estimate in equation~(\ref{eq:est}) and hence $\Gamma$ must contain a bi-proximal element.
\end{proof}

We begin the proof of Proposition~\ref{prop:Ugp} with a lemma that follows easily from the main result in~\cite{P1993}:

\begin{lemma}\cite{P1993}
Suppose $\Gamma \leq \SL(\Cb^{d+1})$ is a subgroup such that every $\gamma \in \Gamma$ is almost unipotent. Let $G$ be the Zariski closure of $\Gamma$ in $\SL(\Rb^{2d+2})$. If $G$ is connected and reductive, then $G$ is compact.
\end{lemma}

\begin{proof}[Proof of Proposition~\ref{prop:Ugp}]
We will induct on $d$. When $d=0$, the proposition is trivial so suppose $d>0$. 

If $G$ is reductive then the above lemma implies that $G$ is compact. Then $G$ is an almost unipotent group with respect to the flag $\{0\} \subsetneq \Cb^{d+1}$. 

If $G$ is not reductive there exists a connected, non-trivial, normal unipotent group $U \leq G$. By Engel's theorem the vector subspace
\begin{align*}
V = \{ v \in \Cb^{d+1} : uv=v \text{ for all $u \in U$}\}
\end{align*}
is non-empty. Since $U$ is non-trivial, $V$ is a proper subspace. Since $U$ is normal in $G$, $G$ preserves the flag $\{0\} \subsetneq V \subsetneq \Cb^d$. Now let $G_1$ be the Zariski closure of $G |_{V}$ and let $\Gamma_1 = \Gamma |_{V}$ then $\Gamma_1$ is Zariski dense in $G_1$. Moreover each element of $\Gamma_1$ is almost unipotent and hence $\Gamma_1 \leq \SL^*(V)$. Since $\Gamma_1$ is Zariski dense in $G_1$ we see that $G_1 \leq \SL^*(V)$. Thus by induction $G_1$ preserves a flag of the form
\begin{align*}
\{0\} \subsetneq V_1 \subsetneq V_2 \subsetneq \dots \subsetneq V_k=V 
\end{align*}
where the image of $G_1$ into $\GL(V_{i+1}/V_i)$ is bounded. 

In a similar fashion let $G_2$ be the Zariski closure of the image of $G$ in $\GL(\Cb^d/V)$ and let $\Gamma_2$ be the image of $\Gamma$ in $\GL(\Cb^d/V)$. Then $\Gamma_2$ will be Zariski dense in $G_2$. Moreover each element of $\Gamma_2$ is almost unipotent and hence $\Gamma_2 \leq \SL^*(\Cb^d/V)$. Since $\Gamma_2$ is Zariski dense in $G_2$ we see that $G_2 \leq \SL^*(\Cb^d/V)$. Thus by induction $G_2$ preserves a flag of the form
\begin{align*}
\{0\}=W_0 \subsetneq W_1 \subsetneq W_2 \subsetneq \dots \subsetneq W_\ell=\Cb^d/V
\end{align*}
where $W_i = V_{k+i}/V$ and the image of  $G_2$ into $\GL(W_{i+1}/W_i)$ is bounded. 

All this implies that $G$ preserves the flag
\begin{align*}
\{0\}=V_0 \subsetneq V_1\subsetneq \dots \subsetneq V_{k+\ell}=\Cb^d
\end{align*}
and the image of  $G$ into $\GL(V_{i+1}/V_i)$ is bounded. Hence we see that $G$ is almost unipotent.
\end{proof}

\begin{proof}[Proof of Proposition~\ref{prop:growth}]
After conjugating $G$, there exists a compact group $K \leq \SU(d+1)$ and a upper triangular group $U \leq \SL(\Cb^{d+1})$ with ones on the diagonal such that $K$ normalizes $U$ and $G \leq KU$. In particular, we can assume $G=KU$. Now for $\varphi \in \End(\Cb^{d+1})$ define $\abs{\varphi} := \sup\{ \abs{u_{i,j}} \}$. Let $\norm{\cdot}$ be the operator norm on $\End(\Cb^{d+1})$ associated to the standard inner product norm on $\Cb^{d+1}$. Then 
\begin{align*}
\norm{k_1 \varphi k_2 } = \norm{\varphi}
\end{align*}
for $k_1,k_2 \in \SU(d+1)$ and $\varphi \in \End(\Cb^{d+1})$. Moreover, since $\norm{\cdot}$ and $\abs{\cdot}$ are norms on $\End(\Cb^{d+1})$ there exists $\alpha >0$ such that 
\begin{align*}
\frac{1}{\alpha}\abs{ \varphi } \leq \norm{\varphi} \leq \alpha \abs{\varphi}
\end{align*}
for all $\varphi \in \End(\Cb^{d+1})$. In particular, for $k \in K$ and $u \in U$ we have 
\begin{align*}
\abs{kuk^{-1}} \leq \alpha \norm{kuk^{-1}} = \alpha\norm{u} \leq \alpha^2 \abs{u}.
\end{align*} 
Now let $g_1, \dots, g_k$ be as in the statement of the proposition. Then $g_i = k_iu_i$ for some $k_i \in K$ and $u_i \in U$. Since $K$ normalizes $U$ we see that
\begin{align*}
g_{i_1} \cdots g_{i_N} = k u_1^\prime \cdots u_N^\prime
\end{align*}
for some $k \in K$ and some $u_k^\prime \in U$ with $\abs{u_k^\prime} \leq \alpha^2 \abs{u_{i_k}}$. Since $\norm{ku} = \norm{u}$ and  $\norm{u} \leq \alpha \abs{u}$ for $k \in K$ and $u \in U$ the proposition will follow from the claim:\newline

\noindent \textbf{Claim:} For any $R>0$, there exists $C=C(R)>0$ such that for any $u_1, \dots, u_N \in U$ with $\abs{u_i} < R$ we have $\abs{ u_1 \cdots u_N} \leq CN^d$. \newline

Now for $i < j$ 
\begin{align*}
(u_1 \cdots u_N)_{i,j} = \sum_{i = a_0 \leq a_1 \leq \dots \leq a_N=j} (u_1)_{ia_1}(u_{2})_{a_1a_2} \cdots (u_N)_{a_{N-1}j}
\end{align*}
Since $U$ is upper triangular with ones on the diagonal at most $d$ terms in the product $(u_1)_{ia_1}(u_{2})_{a_1a_2} \cdots (u_N)_{a_{N-1}j}$ are not equal to one and so
\begin{align*}
\abs{(u_1)_{ia_1}(u_{2})_{a_1a_2} \cdots (u_N)_{a_{N-1}j}} \leq R^d.
\end{align*}
Now we estimate the number of terms in the sum. Notice that $a_{k+1}-a_k \geq 0$ and 
\begin{align*}
\sum_{i=0}^{N-1} a_{k+1}-a_k=j-i.
\end{align*}
Thus we need to estimate the number of ways to write $j-i$ as the sum of $N$ non-negative integers (where order matters). First let $C_n(j-i)$ be the number of ways to write $j-i$ as the sum of $n$ positive integers. Next, at most $j-i$ of the $a_{k+1}-a_k$ are positive and hence the number of ways to write $j-i$ as the sum of $N$ non-negative integers is at most
\begin{align*}
\sum_{n=1}^{j-i} \begin{pmatrix} N \\ n \end{pmatrix}C_n(j-i).
\end{align*}
Then
\begin{align*}
\abs{(u_1 \cdots u_N)_{i,j} }\leq  \sum_{i = a_0 \leq a_1 \leq \dots \leq a_N=j} \abs{(u_1)_{ia_1}(u_{2})_{a_1a_2} \cdots (u_N)_{a_{N-1}j}} \leq R^d \sum_{n=1}^{j-i} \begin{pmatrix} N \\ n \end{pmatrix}C_n(j-i) 
\end{align*}
and since $j-i \leq d$ there exists $C=C(R)>0$ such that 
\begin{align*}
R^d \sum_{n=1}^{j-i} \begin{pmatrix} N \\ n \end{pmatrix} C_n(j-i) < CN^d
\end{align*}
for any $i<j$.
\end{proof}

\section{Constructing additional automorphisms}\label{sec:constr_auto}

Suppose $\Omega$ is a proper $\Cb$-convex open set with $C^1$ boundary. If $\varphi \in \Aut(\Omega)$ is bi-proximal, then we have the following standard form. First let $H^{\pm}$ be the complex tangent hyperplane at $x^{\pm}_\varphi$. Then pick coordinates such that
\begin{enumerate}
\item $x^+_\varphi = [1:0:\dots:0]$,
\item $x^-_\varphi = [0:1:0: \dots :0]$,
\item $H^+ \cap H^- = \{ [0:0:z_2:\dots:z_d]\}$.
\end{enumerate}
With respect to these coordinates, $\varphi$ is represented by a matrix of the form
\begin{align*}
\begin{pmatrix}
\lambda^+ & 0 & ^t\vec{0}\\
0 & \lambda^- & ^t\vec{0} \\
\vec{0} & \vec{0} & A
\end{pmatrix}
\end{align*}
where $A$ is a $(d-1)$-by-$(d-1)$ matrix. Since $H^-=\{[0:z_1:\dots:z_d]\}$ and $\Omega \cap H^- = \emptyset$ we see that $\Omega$ is contained in the affine chart $\Cb^d = \{ [1: z_1: \dots : z_d ] : z_1, \dots, z_d \in \Cb\}$. In this affine chart $x^+_\varphi$ corresponds to $0$ and $T^{\Cb}_0 \partial \Omega = \{0\} \times \Cb^{d-1}$. Then by a projective transformation we may assume that
\begin{enumerate}
\setcounter{enumi}{3}
\item $T_0 \partial \Omega = \Rb \times \Cb^{d-1}$.
\end{enumerate}
Since $\partial \Omega$ is $C^1$ there exists open neighborhoods $V,W \subset \Rb$ of $0$, an open neighborhood $U \subset \Cb^{d-1}$ of $\vec{0}$, and a $C^1$ function $F:V \times U \rightarrow W$ such that if $\Oc = (V+iW) \times U$ then 
\begin{enumerate}
\setcounter{enumi}{4}
\item $\partial \Omega \cap \Oc = \mathrm{Graph}(F)=\left\{ (x+iF\left(x,\vec{z}), \vec{z}\right) : x \in V, \vec{z} \in U\right\}$. 
\end{enumerate}
By another projective transformation we can assume 
\begin{enumerate}
\setcounter{enumi}{5}
\item $\Omega \cap \Oc = \{ (x+iy, \vec{z}) \in \Oc : y > F(x,\vec{z}) \}$.
\end{enumerate}

\begin{theorem}
\label{thm:blow_up}
With the choice of coordinates above,
\begin{align*}
\Omega \cap \left\{ [ z_1: z_2 : 0 : \dots :0] \right\}= \{ [1:z: 0, \dots: 0] : \mathrm{Im}(z) >0\}.
\end{align*}
Moreover for 
\begin{align*}
h = \begin{pmatrix} a & b \\ c & d \end{pmatrix} \in \SL(\Rb^2)
\end{align*}
the projective transformation defined by 
\begin{align*}
\psi_h \cdot [z_1, \dots, z_d] = [ az_1 + bz_2 : cz_1 + d z_2 : z_3 : \dots : z_{d+1}]
\end{align*}
is in $\Aut_0(\Omega)$.
\end{theorem}

\begin{proof}
We can assume $\Oc$ is bounded. Then using part (4) of Theorem~\ref{thm:bi_prox} we can replace $\varphi$ with a power of $\varphi$ so that $\varphi(\Oc) \subset \Oc$.

We first claim that $F(x,\vec{z}) = F(0,\vec{z})$ for $(x,\vec{z}) \in V \times U$. Notice that with our choice of coordinates $\varphi$ acts by
\begin{align*}
\varphi \cdot (z_1, \vec{z}) = \left( \frac{\lambda^-z_1}{\lambda^+}, \frac{A\vec{z}}{\lambda^+}\right)
\end{align*}
where $\lambda^{\pm}$ and $A$ are as above. Since $\varphi$ is bi-proximal 
\begin{align*}
\left(\frac{A}{\lambda^+}\right)^n \rightarrow 0.
\end{align*}
Since $\varphi$ preserves $T_0\partial \Omega = \Rb \times \Cb^{d-1}$ we see that $\lambda^-/\lambda^+ \in \Rb$. Since $x^+_{\varphi}$ is an attracting fixed point we have $\lambda^-/\lambda^+ \in (-1,1)$. Finally since
\begin{align*}
\varphi \cdot (x+iF(x,\vec{z}), \vec{z}) = \left( \frac{\lambda^-}{\lambda^+}x+i\frac{\lambda^-}{\lambda^+}F\left(x,\vec{z}\right), \frac{A}{\lambda^+}\vec{z}\right)
\end{align*}
and $\varphi(\Oc) \subset \Oc$ we see that 
\begin{align*}
F\left(\frac{\lambda^-}{\lambda^+}x,\frac{A}{\lambda^+}\vec{z}\right)=\frac{\lambda^-}{\lambda^+}F(x,\vec{z}).
\end{align*}
Differentiating with respect to $x$ yields
\begin{align*}
(\partial_x F)(x,\vec{z}) = (\partial_x F)\left(\frac{\lambda^-}{\lambda^+}x, \frac{A}{\lambda^+}\vec{z}\right)
\end{align*}
and repeated applications of the above formula shows
\begin{align*}
(\partial_x) F(x,\vec{z}) = (\partial_x F)\left(\left(\frac{\lambda^-}{\lambda^+}\right)^nx, \left(\frac{A}{\lambda^+}\right)^n\vec{z}\right)
\end{align*}
for all $n>0$. Taking the limit as $n$ goes to infinity proves that $(\partial_x F)(x,\vec{z}) = (\partial_x F)(0,0)$. Since $(\partial_x F)(0,0)=0$ we then see that $F(x,\vec{z}) = F(0,\vec{z})$.  

Now for $t \in \Rb$ define the projective map $u_t$ by $u_t \cdot (z_1, \dots, z_d) = (z_1+t, z_2, \dots, z_d)$. Since $F(x,\vec{z}) = F(0,\vec{z})$, we see that there exists $\epsilon >0$ and an open neighborhood $U$ of $0 \in \Cb^{d}$ such that $u_t (z_1, \dots, z_d) \in \Omega$ for all $(z_1,\dots, z_d) \in U \cap \Omega$ and $\abs{t} < \epsilon$. Now by construction
\begin{align*}
u_{(\lambda^- /\lambda^+)t} \circ \varphi = \varphi \circ u_t
\end{align*}
and by part (4) of Theorem~\ref{thm:bi_prox} for any $x \in \Omega$ and $t \in \Rb$ there exist $m$ such that $\varphi^m x \in U$ and $\abs{(\lambda^-/\lambda^+)^m t} < \epsilon$. With this choice of $m$
\begin{align*}
\varphi^m u_t x = u_{(\lambda^-/\lambda^+)^m t} \varphi^m x
\end{align*}
is in $\Omega$. As $\Omega$ is $\varphi$-invariant this implies that $u_t x \in \Omega$. As $x \in \Omega$ and $t \in \Rb$ were arbitrary this implies that $u_t \in \Aut_0(\Omega)$ for all $t \in \Rb$. Also $u_t$ corresponds to the matrix 
\begin{align*}
\begin{pmatrix}
1 & 0 \\ 
t & 1 
\end{pmatrix}
\end{align*}
in the action of $\SL(\Rb^2)$ defined in the statement of the theorem.

The same argument starting with $\varphi^{-1}$ instead of $\varphi$ (that is viewing $\Omega$ as a subset of the affine chart $\{ [z_1 : 1 : z_2 : \dots : z_d]\}$) shows that $\Aut_0(\Omega)$ contains the one-parameter group of automorphisms corresponding to the matrices
\begin{align*}
\begin{pmatrix}
1 & s \\ 
0 & 1 
\end{pmatrix}
\end{align*}
in the action of $\SL(\Rb^2)$ defined in the statement of the theorem.  

Finally it is well known that these two one-parameter subgroups generate all of $\SL(\Rb^2)$ and thus the second part of the theorem is proven. 

It well known that the action of $\SL(\Rb^2)$ restricted to $\{ [1:z:0:\dots : 0] : \textrm{Im}(z) >0\}$ is transitive. As $\Omega$ intersects this set, we see that 
\begin{align*}
\Omega \cap \left\{ [ z_1: z_2 : 0 : \dots :0] \right\}= \{ [1:z: 0, \dots: 0] : \mathrm{Im}(z) >0\}.
\end{align*}

\end{proof}

Notice that the projective map $\psi_t=\psi_{d_t}$ with $d_t = \begin{pmatrix} e^{t} & 0 \\ 0 & e^{-t} \end{pmatrix}$ corresponds to the matrix
\begin{align*}
\begin{pmatrix}
e^{t} & & & & \\
& e^{-t} & & & \\
& & 1 &  & \\
& & & \ddots & \\
& & & & 1 
\end{pmatrix}
\end{align*}
hence $\psi_t$ is bi-proximal, $x^+_{\psi_t} = x^+_{\varphi}$, and $x^-_{\psi_t}=x^-_{\varphi}$. We state this observation (and a little more) as a corollary.

\begin{corollary}
\label{cor:special_bi}
Suppose $\Omega$ is a proper $\Cb$-convex open set with $C^1$ boundary. If $\varphi \in \Aut(\Omega)$ is bi-proximal, then there exists a one-parameter subgroup $\psi_t \in \SL(\Cb^{d+1})$ of bi-proximal elements such that $[\psi_t] \in \Aut_0(\Omega)$ and
\begin{enumerate}
\item $(\psi_t)|_{x^+_{\varphi}} = e^tId|_{x^+_{\varphi}} $,
\item $(\psi_t)|_{x^-_{\varphi}} = e^{-t}Id|_{x^-_{\varphi}} $,
\item $(\psi_t)|_{H^+ \cap H^-} = Id|_{H^+ \cap H^-}$ where $H^{\pm} = T_{x^\pm_{\varphi}}^{\Cb} \partial \Omega$.
\end{enumerate}
\end{corollary}

Since $\SL(\Rb^2)$ acts transitively on $\{ [1:z] \in \Pb(\Cb^2) : \mathrm{Im}(z) >0\}$, Theorem~\ref{thm:blow_up} also implies the following:

\begin{corollary}
\label{cor:bi_2}
Suppose $\Omega$ is a proper $\Cb$-convex open set with $C^1$ boundary, $\varphi \in \Aut(\Omega)$ is bi-proximal, and $L$ is the complex projective line generated by $x^+_{\varphi }$ and $x^-_{\varphi }$. Then for all $p,q \in L \cap \Omega$ distinct there exists $\varphi_{pq} \in \Aut_0(\Omega)$ such that $\varphi_{pq}(p)=q$. 
\end{corollary}

Since $\SL(\Rb^2)$ acts transitively on $\{ [1:z] \in \Pb(\Cb^2) : \mathrm{Im}(z)=0\} \cup \{ [0:1]\}$, Theorem~\ref{thm:blow_up} almost implies the following:

\begin{corollary}
\label{cor:bi_3}
Suppose $\Omega$ is a proper $\Cb$-convex open set with $C^1$ boundary, $\varphi \in \Aut(\Omega)$ is bi-proximal, and $L$ is the complex projective line generated by $x^+_{\varphi}$ and $x^-_{\varphi}$. Then for all $x,y \in L \cap \partial\Omega$ there exists $\varphi_{xy} \in \Aut_0(\Omega)$ such that $\varphi_{xy}(x)=y$.
\end{corollary}

\begin{proof}
Theorem~\ref{thm:blow_up} implies that for all $x,y \in \partial(L \cap \Omega)$ there exists $\varphi \in \text{Aut}_0(\Omega)$ such that $\varphi(x)=y$. Thus we only have to show that $L \cap \partial \Omega = \partial (\Omega \cap L)$. To establish this it is enough to show that $L$ intersects $\partial \Omega$ transversally. Suppose this were not the case, then there exists $x \in L \cap \partial \Omega$ such that $L \subset T_x^{\Cb} \partial \Omega$. Then since $\Omega$ is $\Cb$-convex $L \cap \Omega = \emptyset$ which is nonsense. Thus $L$ intersects $\partial \Omega$ transversally and thus $L \cap \partial \Omega = \partial (\Omega \cap L)$.
\end{proof}

\section{Strict convexity}\label{sec:str}

We call a $\Cb$-convex open set $\Omega$ \emph{strictly $\Cb$-convex} if every complex tangent hyperplane of $\Omega$ intersects $\partial\Omega$ at exactly one point. 

\begin{theorem}
\label{thm:str}
Suppose $\Omega$ is a proper $\Cb$-convex open set with $C^1$ boundary and $\Gamma \leq \PSL(\Cb^d)$ is a torsion-free group dividing $\Omega$. Then $\Omega$ is strictly $\Cb$-convex and for all $x,y \in \partial \Omega$ distinct there exists $\varphi \in \Aut_0(\Omega)$ bi-proximal such that $x=x^+_{\varphi}$ and $y=x^-_{\varphi}$.
\end{theorem}

\begin{remark}
Suppose, for a moment, that $\Gamma$ is a torsion-free word hyperbolic group and $\partial \Gamma$ is the Gromov boundary. Then it is well known that the pairs of attracting and repelling fixed points of elements of $\Gamma$ are dense in $\partial \Gamma \times \partial \Gamma$ (see for instance~\cite[Corollary 8.2.G]{G1987}). The proof of Theorem~\ref{thm:str} follows essentially the same argument, but with the additional technicalities coming from the possible existence of almost unipotent elements and the fact that $\Gamma$ is not necessarily hyperbolic. 
\end{remark}

Theorem~\ref{thm:str} will follow from the next three propositions. 

\begin{proposition}\label{prop:str_1}
 If $x \in \partial \Omega$ then $T_x^{\Cb} \partial \Omega \cap \overline{\{ x^+_{\gamma} : \gamma \in \Gamma \text{ is bi-proximal} \}}$ is non-empty.
\end{proposition}

\begin{proposition}\label{prop:str_2}
If $x,y \in \overline{\{ x^+_{\gamma} : \gamma \in \Gamma \text{ is bi-proximal} \}}$ and $T_x^{\Cb} \partial \Omega \neq T_y^{\Cb} \partial \Omega$ then there exists $\varphi \in \Aut_0(\Omega)$ bi-proximal such that $x=x^+_{\varphi}$ and $y=x^-_{\varphi}$.
\end{proposition}

\begin{proposition}
\label{prop:str_3}
If $x\in \overline{\{ x^+_{\gamma} : \gamma \in \Gamma \text{ is bi-proximal} \}}$ then $T_x^{\Cb} \partial \Omega \cap \partial \Omega = \{x\}$.
\end{proposition}

Delaying the proof of the propositions, we prove Theorem~\ref{thm:str}

\begin{proof}[Proof of Theorem~\ref{thm:str}]
First suppose that $x \in \partial \Omega$, then by Proposition~\ref{prop:str_1} there exists
\begin{align*}
z \in T_x^{\Cb} \partial \Omega \cap \overline{\{ x^+_{\gamma} : \gamma \in \Gamma \text{ is bi-proximal} \}}.
\end{align*}
Since $T_x^{\Cb} \partial \Omega$ is a complex tangent hyperplane containing $z$ and $\partial \Omega$ is $C^1$, $T_z^{\Cb} \partial \Omega = T_x^{\Cb} \partial \Omega$. But by Proposition~\ref{prop:str_3} 
\begin{align*}
 \{z\} = T_z^{\Cb} \partial \Omega \cap \partial \Omega = T_x^{\Cb} \partial \Omega \cap \partial \Omega
\end{align*}
and thus $x=z$. Since $x \in \partial \Omega$ was arbitrary $\partial \Omega =   \overline{\{ x^+_{\gamma} : \gamma \in \Gamma \text{ is bi-proximal} \}}$. Then by Proposition~\ref{prop:str_3}, $\Omega$ is strictly $\Cb$-convex. 

Now suppose $x,y  \in \partial \Omega= \overline{\{ x^+_{\gamma} : \gamma \in \Gamma \text{ is bi-proximal} \}}$ then by Proposition~\ref{prop:str_3} 
\begin{align*}
T_x^{\Cb} \partial \Omega \cap \partial \Omega = \{x\} \text{ and } T_y^{\Cb} \partial \Omega \cap \partial \Omega = \{y\}.
\end{align*}
In particular, if $x \neq y$ then $T_x^{\Cb} \partial \Omega \neq T_y^{\Cb} \partial \Omega$ and so by Proposition~\ref{prop:str_2} there exists $\varphi \in \Aut_0(\Omega)$ is bi-proximal such that $x=x^+_{\varphi}$ and $y=x^-_{\varphi}$. 
\end{proof}

\subsection{Proof of Proposition~\ref{prop:str_1}}

Fix $x \in \partial \Omega$ and $p_n \in \Omega$ such that $p_n \rightarrow x$. Fix a base point $o \in \Omega$. Then since $\Gamma$ acts co-compactly on $\Omega$ there exists $R<+\infty$ and $\varphi_n \in \Gamma$ such that $d_{\Omega}(\varphi_n o, p_n) < R$. By Proposition~\ref{prop:bd_behav}: if $q \in \Omega$ then any limit point of $\{ \varphi_n q\}_{n \in \Nb}$ is in $\partial \Omega \cap T_x^{\Cb} \partial \Omega$.

Now let $\hat{\varphi}_n \in \GL(\Cb^{d+1})$ be representatives of $\varphi_n \in \PSL(\Cb^{d+1})$ such that $\norm{\hat{\varphi}_n}=1$. By passing to a subsequence we may suppose $\hat{\varphi}_n \rightarrow \varphi \in \End(\Cb^{d+1})$. By construction, if $q \in \Pb(\Cb^{d+1})\setminus [\ker \varphi]$ then $\varphi(q) = \lim_{n \rightarrow \infty} \varphi_n (q)$.  In particular $\varphi(\Omega \setminus (\Omega \cap [\ker \varphi])) \subset T_x^{\Cb} \partial \Omega$.  As $\Omega$ is an open set, this implies that $\varphi(\Cb^{d+1}) \subset T_x^{\Cb} \partial \Omega$.  

We now claim that there exists $\gamma \in \Gamma$ bi-proximal such that $x^+_{\gamma} \notin [\ker \varphi]$. To see this, let
\begin{align*}
W:=\textrm{Span}(  x^+_{\gamma} : \gamma \in \Gamma \text{ is bi-proximal} ).
\end{align*}
Since $\phi x^+_{\gamma} = x^+_{\phi \gamma \phi^{-1}}$ we see that $[W]$ is $\Gamma$-invariant. Now suppose $\gamma \in \Gamma$ is bi-proximal, then $x^+_{\gamma} \in [W]$ and so either $[W] \cap \Omega \neq \emptyset$ or $[W] \subset T_{x^+_{\gamma}}^{\Cb} \partial \Omega$. In the latter case
\begin{align*}
[W] \cap \partial \Omega \subset T_{x^+_{\gamma}}^{\Cb} \partial \Omega \cap \partial \Omega = \{x^+_{\gamma}\}
\end{align*}
by Theorem~\ref{thm:bi_prox}. Since $[W]$ contains $x^-_{\gamma}=x^+_{\gamma^{-1}}$ this case is impossible. Thus $[W]\cap \Omega$ is non-empty and so by definition $\Omega^\prime=[W] \cap \Omega$ is a $\Cb$-convex open set in $[W]$. Since $[W]$ is $\Gamma$-invariant, we see that $\Gamma$ acts co-compactly and properly on $\Omega^\prime$. Now it is well known that a proper $\Cb$-convex open set is homeomorphic to an open ball (see for instance~\cite[Theorem 2.4.2]{APS2004}). Thus, by cohomological dimension considerations, we must have that $W=\Cb^{d+1}$.

Since $W=\Cb^{d+1}$ there exists $\gamma \in \Gamma$ bi-proximal such that $x^+_{\gamma} \notin [\ker \varphi]$. Then $\varphi(x^+_{\gamma}) \in T_x^{\Cb} \partial \Omega$ and as 
\begin{align*}
\varphi(x^+_{\gamma}) = \lim_{n \rightarrow\infty} \varphi_n x^+_{\gamma} = \lim_{n \rightarrow\infty}  x^+_{\varphi_n \gamma \varphi_n^{-1}}
\end{align*}
 we see that $T_x^{\Cb} \partial \Omega \cap  \overline{\{ x^+_{\gamma} : \gamma \in \Gamma \text{ is bi-proximal} \}} \neq \emptyset$. 

\subsection{Proof of Proposition~\ref{prop:str_2} and Proposition~\ref{prop:str_3}} We will need to know a little about the action of almost unipotent elements on the boundary $\partial \Omega$.

\begin{lemma}
\label{lem:u_1}
If $u \in \Gamma \setminus \{1\}$ is almost unipotent and $\psi \in \Aut(\Omega)$ is bi-proximal then $x^+_{\psi}$ is not a fixed point of $u$.
\end{lemma}

\begin{remark} The fact that $\Gamma$ is a torsion free discrete group is critical here. \end{remark}

\begin{proof}
Suppose for a contradiction that there exists $\psi \in \Aut(\Omega)$ bi-proximal such that $u(x^+_{\psi}) = x^+_{\psi}$. Let $x^\pm := x^\pm_{\psi}$ and let $L$ be the complex projective line containing $x^+$ and $x^-$. Let $H^\pm$ be the complex tangent hyperplane to $\Omega$ at $x^\pm$. By Theorem~\ref{thm:blow_up} there exists coordinates such that
\begin{enumerate}
\item $x^+ = [1:0:\dots:0]$,
\item $x^- = [0:1:0: \dots :0]$,
\item $H^+ \cap H^- = \{ [0:0:z_1:\dots:z_{d-1}] : z_1, \dots, z_{d-1} \in \Cb\}$,
\item $ \Omega \cap L = \{ [ 1: z : 0 : \dots : 0] : \textrm{Im}(z) >0\}$,
\end{enumerate}
and $\Aut_0(\Omega)$ contains the automorphisms
\begin{align*}
a_t \cdot [z_1 : z_2 : \dots : z_{d+1}] = [e^t z_1 : e^{-t} z_2 : z_3 : \dots : z_{d+1}].
\end{align*}
Since $u$ fixes $x^+$ it also fixes $H^+=T_{x^+}^{\Cb} \partial \Omega$ and hence with respect to these coordinates $u$ is represented by a matrix of the form:
\begin{align*}
\begin{pmatrix}
1 & b & \vec{x}^t \\
0 & c & \vec{0}^t \\
\vec{0} & \vec{y} & A 
\end{pmatrix}.
\end{align*}
where $c \in \Cb$, $\vec{x},\vec{y} \in \Cb^{d-1}$, and $A$ is a $(d-1)$-by-$(d-1)$ matrix. Now a calculation shows that $u^\prime=\lim_{t \rightarrow \infty} a_{-t} u a_t$ exists in $\PSL(\Cb^{d+1})$ and  is represented by a matrix of the form:
\begin{align*}
\begin{pmatrix}
1 & 0 & 0\\
0 & c & \vec{0}^t \\
\vec{0} & \vec{0} & A 
\end{pmatrix}.
\end{align*}
Since $\Aut(\Omega)$ is closed in $\PSL(\Cb^d)$ we have that $u^\prime \in \Aut(\Omega)$. Since $u^\prime$ is the limit of almost unipotent elements, $u^\prime$ is almost unipotent and so $\abs{c}=1$. Since $u$ leaves $\Omega \cap L$ invariant $c \in \Rb$. Then by possibly replacing $u$ with $u^2$ we may assume that $c=1$. Then $u^\prime(z)=z$ for all $z \in L \cap \Omega$. Fix some $z \in \Omega \cap L$, then we have
\begin{align*}
\inf_{p \in \Omega} d_{\Omega}\left(u (p),p\right) \leq \lim_{t \rightarrow \infty} d_{\Omega}\left(u(a_t z), a_t z\right) = \lim_{t \rightarrow \infty} d_{\Omega}\left((a_{-t}ua_t)( z), z\right) =d_{\Omega}\left(u^\prime(z), z\right)=0.
\end{align*}
This contradicts Corollary~\ref{cor:inj_rad}. 
\end{proof}

We can use a standard argument to construct bi-proximal elements.
 
\begin{lemma}
\label{lem:str_1}
Suppose $\phi, \gamma \in \Gamma$ are bi-proximal. If $x^+_{\gamma}, x^-_{\gamma}, x^+_{\phi}, x^-_{\phi}$ are all distinct and $U^+, W^+$ are neighborhoods of $x^+_{\gamma}, x^+_{\phi}$ in $\overline{\Omega}$ then there exists $\psi \in \Gamma$ bi-proximal such that $x^+_{\psi} \in U^+$ and $x^-_{\psi} \in W^+$.
\end{lemma}

\begin{proof}
Pick neighborhoods $U^-,W^-$ of $x^-_{\gamma}, x^-_{\phi}$ in $\overline{\Omega}$ and by possibly shrinking $U^\pm, W^\pm$ assume that $\overline{U^+},\overline{U^-},\overline{W^+},\overline{W^-}$ are all disjoint. Let $\Oc= W^+ \cup W^- \cup U^+ \cup U^-$. Then by Theorem~\ref{thm:bi_prox}, there exists $m$ such that $\gamma^{ \mp m}( \Oc \setminus U^\pm) \subset U^\mp$ and $\phi^{ \mp m}(\Oc \setminus W^\pm) \subset W^\mp$. So if $\psi = \gamma^m\phi^{-m}$ then $\psi(U^+) \subset U^+$ and $\psi^{-1}(W^+) \subset W^+$. By Theorem~\ref{thm:bi_prox}, $\psi$ is either almost unipotent or bi-proximal. Moreover by Lemma~\ref{lem:weak_attract}, $E^{\pm}(\psi) = x^{\pm}$ for some $x^\pm \in \partial \Omega$. Since $\psi(U^+) \subset U^+$ and $\psi^{-1}(W^+) \subset W^+$,  part (1) of Proposition~\ref{prop:attracting} implies that $x^+ \in \overline{U^+}$ and $x^- \in \overline{W^+}$. Since $\overline{U^+}$ and $\overline{W^+}$ are disjoint this implies that $x^+ \neq x^-$. Thus $\psi$ is not almost unipotent.
\end{proof}

The next lemma constructs even more bi-proximal elements.

\begin{lemma}
\label{lem:str_4}
Suppose $x_1,\dots, x_m \in \partial \Omega$ are distinct, then there exists $\gamma \in \Gamma$ bi-proximal such that $x_1,\dots, x_m,x^+_{\gamma}, x^-_{\gamma}$ are all distinct.
\end{lemma}

\begin{proof}\
First suppose that $\Gamma$ contains a non-trivial almost unipotent element $u$. Let $\phi \in \Gamma$ be bi-proximal. Since $x^+_{\phi}$ is a not fixed point of any power of $u$, $u^n(x^+_{\phi}) = u^m( x^+_{\phi})$ if and only if $m=n$. Thus $x_1,\dots, x_m$ each appears at most once in the list 
\begin{align*}
x^+_{\phi}, \ u(x^+_{\phi}), \ u^2(x^+_{\phi}), \ u^3(x^+_{\phi}), \dots
\end{align*}
In a similar fashion, $x_1,\dots, x_m$ each appears at most once in the list 
\begin{align*}
x^-_{\phi}, \ u(x^-_{\phi}), \ u^2(x^-_{\phi}), \ u^3(x^-_{\phi}), \dots
\end{align*}
Thus for $N$ sufficiently large $x_1,\dots,x_m,u^N (x^+_{\phi}), u^N(x^-_{\phi})$ are all distinct and so $\gamma = u^N \phi u^{-N}$ satisfies the conclusion of the lemma.

Otherwise every element of $\Gamma \setminus \{1\}$ is bi-proximal. In this case, if there does not exist  $\gamma \in \Gamma \setminus \{1\}$ bi-proximal such that $x_1,\dots, x_m,x^+_{\gamma}, x^-_{\gamma}$ are all distinct then 
\begin{align*}
\Gamma = \cup_{i=1}^m \textrm{Stab}_{\Gamma}(x_i)
\end{align*}
and at least one of $\mathrm{Stab}_{\Gamma}(x_1), \dots, \mathrm{Stab}_{\Gamma}(x_m)$ has finite index in $\Gamma$ (see for instance~\cite[Lemma 4.1]{N1954}). So by passing to a finite index subgroup and possibly relabeling we can assume that $\Gamma$ fixes $x_1$. 

Now let $\Gamma^\prime$ be the preimage of $\Gamma$ under the map $\SL(\Cb^{d+1}) \rightarrow \PSL(\Cb^{d+1})$. Then we can conjugate $\Gamma^\prime$ to be a subset of the matrices of the form:
\begin{align*}
\begin{pmatrix}
\lambda & ^t\vec{v}\\
0 & A 
\end{pmatrix}
\end{align*}
where $\lambda \in \Cb^*$, $\vec{v} \in \Cb^d$ and $A$ is a $d$-by-$d$ matrix. 

We claim that $\Gamma^\prime$ is commutative. Suppose for a contradiction that $\gamma \in [\Gamma^\prime, \Gamma^\prime]$ is non-trivial.  Then $x_1$ is an eigenline of $\gamma$ with eigenvalue  one. Since $\gamma$ is non-trivial $\gamma$ is bi-proximal (by assumption). Then by part (4) of Theorem~\ref{thm:bi_prox} the only fixed points of $\gamma$ in $\partial \Omega$ are $x^+_{\gamma}$ and $x^-_{\gamma}$. Since $x_1 \in \partial \Omega$ is a fixed point with eigenvalue one and $\gamma \in \SL(\Cb^{d+1})$ we have a contradiction. So $\Gamma^\prime$ and hence $\Gamma$ is commutative. 

Now fix $\gamma_0 \in \Gamma \setminus \{1\}$ and let $x^\pm := x^\pm_{\gamma_0}$. Since $\Gamma$ is commutative and the only fixed points of $\gamma_0$ in $\partial \Omega$ are $x^{\pm}$ we see that $\Gamma \cdot \{x^+, x^-\} =  \{x^+, x^-\}$. Since the only fixed points of $\gamma \in \Gamma \setminus \{1\}$ are $x^+_{\gamma}$ and $x^-_{\gamma}$ we have that
\begin{align*}
\{x^+,x^-\}= \{x^+_{\gamma}, x^-_{\gamma} \}.
\end{align*}
Using Theorem~\ref{thm:blow_up} we can pick coordinates such that $x^+=[1:0:\dots:0]$, $x^-=[0:1:0:\dots:0]$, and if $L$ is complex projective line containing $x^+$ and $x^-$ then
\begin{align*}
\Omega \cap L = \{ [1:z:0:\dots:0] : \operatorname{Im}(z) >0\}.
\end{align*}
Since $L$ is $\Gamma$-invariant and $\Gamma \backslash \Omega$ is compact, $\Gamma$ acts co-compactly on $\Omega \cap L$. Also $\Gamma$ acts by isometries on $(\Omega \cap L, d_{\Omega \cap L})$ which by Proposition~\ref{prop:proj_disk} is isometric to hyperbolic real 2-space. But this contradicts the fact that $\Gamma$ is commutative.
\end{proof}

We can now prove Proposition~\ref{prop:str_2} and Proposition~\ref{prop:str_3}.

\begin{proof}[Proof of Proposition~\ref{prop:str_2}]
Suppose $\{\gamma_n\}_{n \in \Nb}, \{\phi_n\}_{n \in \Nb} \subset \Gamma$ are sequences of bi-proximal elements such that $x^+_{\gamma_n} \rightarrow x$ and $x^+_{\phi_n} \rightarrow y$. By passing to subsequences we can assume $x^-_{\gamma_n} \rightarrow x^\prime$ and $x^-_{\phi_n} \rightarrow y^\prime$. 

First suppose that $x,x^\prime, y,y^\prime$ are all distinct. Then for $n$ large $x^+_{\gamma_n}, x^-_{\gamma_n}, x^+_{\phi_n}, x^-_{\phi_n}$ are all distinct and so using Lemma~\ref{lem:str_1} we can find a sequence $\{\psi_n\} \subset \Gamma $ of bi-proximal elements such that $x^+_{\psi_n} \rightarrow x$ and $x^-_{\psi_n} \rightarrow y$. Now let $H_x := T_x^{\Cb} \partial \Omega$, $H_y := T_y^{\Cb} \partial \Omega$, and  $H_n^{\pm}: = T_{x^{\pm}_{\psi_n}}^{\Cb} \partial \Omega$. Since $\partial \Omega$ is $C^1$ and $H_x \neq H_y$,  for $n$ large enough $H_n^+ \neq H_n^-$ and in the of space complex codimension two subspaces $H_n^+ \cap H_n^- \rightarrow H_x \cap H_y$. By Corollary~\ref{cor:special_bi}, there exist bi-proximal elements $\hat{\psi}_n \in \SL(\Cb^{d+1})$ such that 
\begin{enumerate}
\item $[\hat{\psi}_n] \in \Aut_0(\Omega)$, 
\item $\hat{\psi}_n|_{x^+_{\psi_n}} = 2Id|_{x^+_{\psi_n}}$, 
\item $\hat{\psi}_n|_{x^-_{\psi_n}} = (1/2)Id|_{x^-_{\psi_n}}$, and 
\item $\hat{\psi}_n|_{H_n^+ \cap H_n^-} = Id|_{H_n^+ \cap H_n^-}$. 
\end{enumerate}
Now since $x^+_{\psi_n} \rightarrow x$, $x^-_{\psi_n} \rightarrow y$, and $H_n^+ \cap H_n^- \rightarrow H_x \cap H_y$, we see that $\hat{\psi}_n$ converges to $\hat{\psi} \in \SL(\Cb^{d+1})$ such that $\hat{\psi}|_{x}=2Id|_{x}$, $\hat{\psi}|_{y}=(1/2)Id|_{y}$, and $\hat{\psi}|_{H_x \cap H_y} = Id|_{H_x \cap H_y}$. As $\Aut_0(\Omega) \subset \PSL(\Cb^{d+1})$ is closed, $[\hat{\psi}] \in \Aut_0(\Omega)$. This establishes the lemma in this special case.

Next consider case in which $x,x^\prime, y,y^\prime$ are not all distinct. By Lemma~\ref{lem:str_4} there exists $\varphi \in \Gamma$ bi-proximal such that $x^+_{\varphi}$ and $x^-_{\varphi}$ are not in the set $\{x,y,x^\prime, y^\prime\}$. Then for $n$ large $x^+_{\gamma_n}, x^-_{\gamma_n}, x^+_{\varphi}, x^-_{\varphi}$ are all distinct and we can use Lemma~\ref{lem:str_1} to find a sequence $\{ \gamma_n^\prime \} \subset \Gamma$ of bi-proximal elements such that $x^+_{\gamma_n^\prime} \rightarrow x$ and $x^-_{\gamma_n^\prime} \rightarrow x^+_{\varphi}$. So by replacing $\gamma_n$ with $\gamma_n^\prime$ we may suppose $x^+_{\gamma_n} \rightarrow x$ and $x^-_{\gamma_n} \rightarrow x^+_{\varphi}$. A similar argument shows that we may assume $x^+_{\phi_n} \rightarrow y$ and $x^-_{\phi_n} \rightarrow x^-_{\varphi}$. Since $x,y,x^+_{\varphi}, x^-_{\varphi}$ are all distinct we can now apply the argument above.
\end{proof}

\begin{proof}[Proof of Proposition~\ref{prop:str_3}]
Suppose $\{\gamma_n\}_{n \in \Nb} \subset \Gamma$ is a sequence of bi-proximal elements such that $x^+_{\gamma_n} \rightarrow x$. By passing to subsequences we can assume $x^-_{\gamma_n} \rightarrow x^\prime$. Now pick $\gamma \in \Gamma$ bi-proximal such that $x,x^\prime, x^+_{\gamma}, x^-_{\gamma}$ are all distinct. Then using Lemma~\ref{lem:str_1} we can find a sequence of bi-proximal elements $\{ \phi_n \}_{n \in \Nb} \subset \Gamma$ such that $x^+_{\phi_n} \rightarrow x$ and $x^-_{\phi_n} \rightarrow x^+_{\gamma}$. 

By Theorem~\ref{thm:bi_prox}, $T_{x^+_{\gamma}}^{\Cb} \partial \Omega \cap \partial \Omega = \{x^+_{\gamma}\}$ and in particular $T_{x^+_{\gamma}}^{\Cb} \partial \Omega \neq T_{x}^{\Cb} \partial \Omega$. Thus by Proposition~\ref{prop:str_2}, there exists $\varphi \in \Aut_0(\Omega)$ bi-proximal such that $x=x^+_{\varphi}$ and $x^+_{\gamma}=x^-_{\varphi}$. Then by Theorem~\ref{thm:bi_prox} we see that $T_x^{\Cb} \partial \Omega \cap \partial \Omega = \{x\}$.
\end{proof}

\section{The boundary is smooth}\label{sec:smth}

The purpose of this section is to prove the following.

\begin{theorem}\label{thm:smth}
Suppose $\Omega$ is a divisible proper $\Cb$-convex open set with $C^1$ boundary. Then $\partial \Omega$ is a $C^\infty$ embedded submanifold of $\Pb(\Cb^{d+1})$. 
\end{theorem}

Theorem~\ref{thm:smth} will follow immediately from the next two lemmas

\begin{lemma}
Suppose $G$ is a connected Lie group acting smoothly on a smooth manifold $M$. Then an orbit $G\cdot m$ is an embedded smooth submanifold of $M$ if and only if $G \cdot m$ is locally closed in $M$. 
\end{lemma}

Here smooth mean $C^\infty$ and for a proof see~\cite[Theorem 15.3.7]{tD2008}. Since $\partial \Omega \subset \Pb(\Cb^{d+1})$ is closed, Theorem~\ref{thm:smth} follows from:

\begin{lemma}
If $\Omega$ satisfies the hypothesis of Theorem~\ref{thm:smth}, then $\Aut_0(\Omega)$ acts transitively on $\partial \Omega$.
\end{lemma}

\begin{proof}
Suppose $x,y \in \partial \Omega$ are distinct and $L$ is the complex projective line containing $x$ and $y$. By Theorem~\ref{thm:str} $\Omega$ is strictly convex and so $L$ intersects $\Omega$. Then by Corollary~\ref{cor:bi_3} there exists $\phi \in \Aut_0(\Omega)$ such that $\phi(x)=y$.
\end{proof}

\section{Completing the proof of Theorem~\ref{thm:main}}\label{sec:complete}

Before finishing the proof of Theorem~\ref{thm:main} we will need one more lemma which follows immediately from Theorem~\ref{thm:blow_up} and Theorem~\ref{thm:str}.

\begin{lemma}\label{lem:proj_disk}
Suppose $\Omega$ is a divisible proper $\Cb$-convex open set with $C^1$ boundary. If $L$ is a complex projective line that intersects $\Omega$ then $\Omega \cap L$ is a projective disk.
\end{lemma}

Now suppose $\Omega$ is a divisible proper $\Cb$-convex open set with $C^1$ boundary. We will show that $\Omega$ is a projective ball. First, using Theorem~\ref{thm:blow_up}, we can find an affine chart $\Cb^d$ containing $\Omega$ such that  $0 \in \partial \Omega$, $T_0 \partial \Omega = \Rb \times \Cb^{d-1}$, and $\Aut_0(\Omega)$ contains the 1-parameter subgroup
\begin{align*}
a_t \cdot (z_1, \vec{z}) = ( e^{-2t} z_1, e^{-t} \vec{z}).
\end{align*}
Since $\partial \Omega$ is $C^\infty$ there exists open neighborhoods $V,W \subset \Rb$ of $0$, an open neighborhood $U \subset \Cb^{d-1}$ of $\vec{0}$, and a $C^\infty$ function $F:V \times U \rightarrow W$ such that if $\Oc = (V+iW) \times U$ then 
\begin{align*}
\partial \Omega \cap \Oc = \mathrm{Graph}(F)=\left\{ (x+iF\left(x,\vec{z}), \vec{z}\right) : x \in V, \vec{z} \in U\right\}. 
\end{align*}
By another projective transformation we can assume 
\begin{align*}
\Omega \cap \Oc = \{ (x+iy, \vec{z}) \in \Oc : y > F(x,\vec{z}) \}
\end{align*}
Now since $\partial \Omega$ is $a_t$ invariant we see that
\begin{align*}
F(x,\vec{z}) = e^{2t} F(e^{-2t} x, e^{-t}\vec{z}).
\end{align*}
Taking the Hessian yields
\begin{align*}
\mathrm{Hess}(F)_{(x,\vec{z})} \Big( (X_1,Z_1), (X_2,Z_2) \Big) 
&= e^{2t} \mathrm{Hess}(F)_{(e^{-2t}x,e^{-t}\vec{z})} \Big( (e^{-2t}X_1, e^{-t}Z_1) , (e^{-2t}X_2, e^{-t}Z_2) \Big)\\
&=\mathrm{Hess}(F)_{(e^{-2t}x,e^{-t}\vec{z})} \Big( (e^{-t}X_1, Z_1) , (e^{-t}X_2, Z_2) \Big).
\end{align*}
Then sending $t \rightarrow \infty$ shows that
\begin{align*}
\mathrm{Hess}(F)_{(x,\vec{z})}  \Big( (X_1,Z_1), (X_2,Z_2) \Big)  = \mathrm{Hess}(F)_{(0,0)}  \Big( (0,Z_1), (0,Z_2) \Big).
\end{align*}
Thus 
\begin{align*}
F(x,\vec{z}) = \frac{1}{2} \mathrm{Hess}(F)_{(0,0)}  \Big( (0,\vec{z}), (0,\vec{z}) \Big)
\end{align*}
for all $x \in V$ and $\vec{z} \in U$. Since $\Omega$ is strictly $\Cb$-convex and $\{0\} \times \Cb^{d-1}$ is the complex tangent hyperplane at $0$, $\Omega \cap\left( \{0\} \times \Cb^{d-1} \right)= \emptyset$. Thus $F(0,\vec{z}) >0$ for all $\vec{z} \in U$ and so the Hessian of $F$ is positive definite on the complex tangent space. 

Now let 
\begin{align*}
H(\vec{z}) = \frac{1}{2} \text{Hess}(F)_{(0, 0)}\Big( (0,\vec{z}),(0,\vec{z})\Big).
\end{align*}
Then
\begin{align*}
 \Omega \cap \Oc= \{ (z_1,\vec{z}) \in \Oc: \text{Im}(z_1) > H(\vec{z}) \}
\end{align*}
and by part (4) of Theorem~\ref{thm:bi_prox}
\begin{align*}
\Omega = \bigcup_{t \in \Rb} a_t (\Oc \cap \Omega).
\end{align*}
Also, if $(w_1, \vec{w}) = a_t (z_1,\vec{z})$ then $\text{Im}(z_1) > H(\vec{z}) \text{ if and only if } \text{Im}(w_1) > H(\vec{w})$. So
\begin{align*}
\Omega = \left\{ (z_1,\vec{z}) \in \Cb^d: \text{Im}(z_1) > H(\vec{z}) \right\}.
\end{align*}

Since the Hessian is a symmetric $\Rb$-linear form
\begin{align*}
 \mathrm{Hess}(F)_{(0,0)}  \Big( (0,\vec{z}), (0,\vec{w}) \Big)  = \sum_{i,j=2}^d \left(A_{i,j} z_i \overline{w}_j + B_{i,j} \overline{z}_i w_j + C_{i,j} z_i w_j + D_{i,j} \overline{z}_i \overline{w}_j\right)
 \end{align*}
 for some real symmetric matrices $(A_{i,j}), (B_{i,j}), (C_{i,j}), (D_{i,j})$. Since the Hessian is real valued, $(A_{i,j})=(B_{i,j})$ and $(C_{i,j})=(D_{i,j})$. In particular we can write $H = \Lc + \textrm{Re}(\Qc)$ where 
\begin{align*}
\Lc(z_2, \dots, z_d) = \sum_{i,j=2}^d \alpha_{i,j} z_i \overline{z_j} \text{ and } \Qc(z_2, \dots, z_d) = \textrm{Re} \left(\sum_{i,j=2}^d \beta_{i,j} z_i z_j \right).
\end{align*}
and $(\alpha_{i,j})$ and $(\beta_{i,j})$ are real symmetric matrices. 
 
We next claim that there exists a basis of $\Cb^{d-1}$ such that $H(\vec{z},\vec{w})=\sum \abs{z_i}^2$ which would imply that $\Omega$ is a projective ball. 

Since $H(\vec{z}) > 0$, $\Lc(e^{i\theta}\vec{z}) = \Lc(\vec{z})$, and $\Qc(e^{i\theta}\vec{z}) = e^{2i\theta}\Qc(\vec{z})$ we see that 
\begin{align*}
\Lc(\vec{z}) > \abs{\Qc(\vec{z})} \geq 0
\end{align*}
for all $\vec{z} \in \Cb^{d-1}$. Thus there exists an orthogonal basis of $\Cb^{d-1}$ such that $\Lc(\vec{z}) = \sum_{i=2}^d \abs{z_i}^2$. Since $(\beta_{ij})$ is a symmetric real matrix there exists an orthonormal basis of $\Cb^{d-1}$ such that
\begin{align*}
\Qc(z_2, \dots, z_d) = \sum_{j=2}^d \beta_{j} z_j^2
\end{align*}
with $\beta_j \geq 0$. With these choices
\begin{align*}
H(\vec{z}) = \sum_{j=2}^d \abs{z_j}^2 + \textrm{Re}\left( \sum_{j=2}^d \beta_j z_j^2\right).
\end{align*}
Since $\Lc(\vec{z}) > \abs{\Qc(\vec{z})}$ we see that $\beta_j < 1$. 

Let $e_1 = (1,0,\dots,0), e_2=(0,1,0,\dots, 0), \dots, e_d = (0,\dots,0,1)$ be the standard basis of $\Cb^d$. For $\epsilon >0$ and $2 \leq j \leq d$, consider the complex line $L^\prime$ parametrized by $z \rightarrow (i\epsilon) e_1 + ze_j$. Since $\epsilon$ is positive,  $(i\epsilon) e_1 \in L^\prime \cap \Omega$  and hence by Lemma~\ref{lem:proj_disk} $L^\prime \cap \Omega$ is a projective disk. Also $(i\epsilon) e_1 + ze_j \in L^\prime \cap \partial \Omega$ if and only if 
\begin{align*}
\epsilon = H(ze_j) = \abs{z}^2 + \textrm{Re}(\beta_j z^2) =  (1 + \beta_j) \textrm{Re}(z)^2 + (1-\beta_j) \textrm{Im}(z)^2.
\end{align*}
Now this describes the boundary of a projective disk if and only if $\beta_j = 0$. Thus we have that 
\begin{align*}
H(\vec{z}) = \sum_{j=2}^d \abs{z_j}^2
\end{align*}
and
\begin{align*}
\Omega = \left\{ (z_1,\vec{z}) \in \Cb^d: \text{Im}(z_1) > \sum_{i=2}^{d} \abs{z_i}^2 \right\}.
\end{align*}
Thus $\Omega$ is a projective ball.

\bibliographystyle{alpha}
\bibliography{hilbert}

\begin{thebibliography}{Won77}

\bibitem[APS04]{APS2004}
Mats Andersson, Mikael Passare, and Ragnar Sigurdsson.
\newblock {\em Complex convexity and analytic functionals}, volume 225 of {\em
  Progress in Mathematics}.
\newblock Birkh\"auser Verlag, Basel, 2004.

\bibitem[Bea98]{B1998}
A.~F. Beardon.
\newblock The {A}pollonian metric of a domain in {${\bf R}^n$}.
\newblock In {\em Quasiconformal mappings and analysis ({A}nn {A}rbor, {MI},
  1995)}, pages 91--108. Springer, New York, 1998.

\bibitem[Ben04]{B2004}
Yves Benoist.
\newblock Convexes divisibles. {I}.
\newblock In {\em Algebraic groups and arithmetic}, pages 339--374. Tata Inst.
  Fund. Res., Mumbai, 2004.

\bibitem[Ben05]{B2005}
Yves Benoist.
\newblock Convexes divisibles. {III}.
\newblock {\em Ann. Sci. \'Ecole Norm. Sup. (4)}, 38(5):793--832, 2005.

\bibitem[Ben08]{B2008}
Yves Benoist.
\newblock A survey on divisible convex sets.
\newblock In {\em Geometry, analysis and topology of discrete groups}, volume~6
  of {\em Adv. Lect. Math. (ALM)}, pages 1--18. Int. Press, Somerville, MA,
  2008.

\bibitem[dlH00]{dlP2000}
Pierre de~la Harpe.
\newblock {\em Topics in geometric group theory}.
\newblock Chicago Lectures in Mathematics. University of Chicago Press,
  Chicago, IL, 2000.

\bibitem[Dub09]{D2009}
Lo{\"{\i}}c Dubois.
\newblock Projective metrics and contraction principles for complex cones.
\newblock {\em J. Lond. Math. Soc. (2)}, 79(3):719--737, 2009.

\bibitem[Fra89]{F1989}
Sidney Frankel.
\newblock Complex geometry of convex domains that cover varieties.
\newblock {\em Acta Math.}, 163(1-2):109--149, 1989.

\bibitem[GH00]{GH2000}
F.~W. Gehring and K.~Hag.
\newblock The {A}pollonian metric and quasiconformal mappings.
\newblock In {\em In the tradition of {A}hlfors and {B}ers ({S}tony {B}rook,
  {NY}, 1998)}, volume 256 of {\em Contemp. Math.}, pages 143--163. Amer. Math.
  Soc., Providence, RI, 2000.

\bibitem[Gol88]{G1988}
William~M. Goldman.
\newblock Geometric structures on manifolds and varieties of representations.
\newblock In {\em Geometry of group representations ({B}oulder, {CO}, 1987)},
  volume~74 of {\em Contemp. Math.}, pages 169--198. Amer. Math. Soc.,
  Providence, RI, 1988.

\bibitem[Gol90]{G1990}
William~M. Goldman.
\newblock Convex real projective structures on compact surfaces.
\newblock {\em J. Differential Geom.}, 31(3):791--845, 1990.

\bibitem[Gol99]{G1999}
William~M. Goldman.
\newblock {\em Complex hyperbolic geometry}.
\newblock Oxford Mathematical Monographs. The Clarendon Press Oxford University
  Press, New York, 1999.
\newblock Oxford Science Publications.

\bibitem[Gro87]{G1987}
M.~Gromov.
\newblock Hyperbolic groups.
\newblock In {\em Essays in group theory}, volume~8 of {\em Math. Sci. Res.
  Inst. Publ.}, pages 75--263. Springer, New York, 1987.

\bibitem[GT87]{GT1987}
M.~Gromov and W.~Thurston.
\newblock Pinching constants for hyperbolic manifolds.
\newblock {\em Invent. Math.}, 89(1):1--12, 1987.

\bibitem[Guo13]{G2013}
Ren Guo.
\newblock Characterizations of hyperbolic geometry among hilbert geometries: A
  survey.
\newblock In {\em Handbook of Hilbert geometry}. European Mathematical Society
  Publishing, to appear in 2013.

\bibitem[H{\"o}r07]{H2007}
Lars H{\"o}rmander.
\newblock {\em Notions of convexity}.
\newblock Modern Birkh\"auser Classics. Birkh\"auser Boston Inc., Boston, MA,
  2007.
\newblock Reprint of the 1994 edition.

\bibitem[IK99]{IK1999}
A.~V. Isaev and S.~G. Krantz.
\newblock Domains with non-compact automorphism group: a survey.
\newblock {\em Adv. Math.}, 146(1):1--38, 1999.

\bibitem[JM87]{JM1987}
Dennis Johnson and John~J. Millson.
\newblock Deformation spaces associated to compact hyperbolic manifolds.
\newblock In {\em Discrete groups in geometry and analysis ({N}ew {H}aven,
  {C}onn., 1984)}, volume~67 of {\em Progr. Math.}, pages 48--106. Birkh\"auser
  Boston, Boston, MA, 1987.

\bibitem[Kap07]{K2007}
Michael Kapovich.
\newblock Convex projective structures on {G}romov-{T}hurston manifolds.
\newblock {\em Geom. Topol.}, 11:1777--1830, 2007.

\bibitem[Kli11]{K2011}
B.~Klingler.
\newblock Local rigidity for complex hyperbolic lattices and {H}odge theory.
\newblock {\em Invent. Math.}, 184(3):455--498, 2011.

\bibitem[Kos68]{K1968}
J.-L. Koszul.
\newblock D\'eformations de connexions localement plates.
\newblock {\em Ann. Inst. Fourier (Grenoble)}, 18(fasc. 1):103--114, 1968.

\bibitem[Kra13]{K2013}
Steven~G. Krantz.
\newblock The impact of the theorem of bun wong and rosay.
\newblock {\em Complex Variables and Elliptic Equations}, to appear in 2013.

\bibitem[Mar13]{M2013}
Ludovic Marquis.
\newblock Around groups in hilbert geometry.
\newblock In {\em Handbook of Hilbert geometry}. European Mathematical Society
  Publishing, to appear in 2013.

\bibitem[Neu54]{N1954}
B.~H. Neumann.
\newblock Groups covered by permutable subsets.
\newblock {\em J. London Math. Soc.}, 29:236--248, 1954.

\bibitem[NPZ08]{NPZ2008}
Nikolai Nikolov, Peter Pflug, and W{\l}odzimierz Zwonek.
\newblock An example of a bounded {${\bf C}$}-convex domain which is not
  biholomorphic to a convex domain.
\newblock {\em Math. Scand.}, 102(1):149--155, 2008.

\bibitem[Pra94]{P1993}
Gopal Prasad.
\newblock {${\bf R}$}-regular elements in {Z}ariski-dense subgroups.
\newblock {\em Quart. J. Math. Oxford Ser. (2)}, 45(180):541--545, 1994.

\bibitem[Qui10]{Q2010}
Jean-Fran{\c{c}}ois Quint.
\newblock Convexes divisibles (d'apr\`es {Y}ves {B}enoist).
\newblock {\em Ast\'erisque}, (332):Exp. No. 999, vii, 45--73, 2010.
\newblock S{\'e}minaire Bourbaki. Volume 2008/2009. Expos{\'e}s 997--1011.

\bibitem[Ros79]{R1979}
Jean-Pierre Rosay.
\newblock Sur une caract\'erisation de la boule parmi les domaines de {${\bf
  C}^{n}$} par son groupe d'automorphismes.
\newblock {\em Ann. Inst. Fourier (Grenoble)}, 29(4):ix, 91--97, 1979.

\bibitem[SM02]{SM2002}
Edith Soci{\'e}-M{\'e}thou.
\newblock Caract\'erisation des ellipso\"\i des par leurs groupes
  d'automorphismes.
\newblock {\em Ann. Sci. \'Ecole Norm. Sup. (4)}, 35(4):537--548, 2002.

\bibitem[Tay11]{T2011}
Michael~E. Taylor.
\newblock {\em Partial differential equations {I}. {B}asic theory}, volume 115
  of {\em Applied Mathematical Sciences}.
\newblock Springer, New York, second edition, 2011.

\bibitem[tD08]{tD2008}
Tammo tom Dieck.
\newblock {\em Algebraic topology}.
\newblock EMS Textbooks in Mathematics. European Mathematical Society (EMS),
  Z\"urich, 2008.

\bibitem[Vey70]{V1970}
Jacques Vey.
\newblock Sur les automorphismes affines des ouverts convexes saillants.
\newblock {\em Ann. Scuola Norm. Sup. Pisa (3)}, 24:641--665, 1970.

\bibitem[Won77]{W1977}
B.~Wong.
\newblock Characterization of the unit ball in {${\bf C}^{n}$} by its
  automorphism group.
\newblock {\em Invent. Math.}, 41(3):253--257, 1977.

\end{thebibliography}

\end{document}